\documentclass[12pt,reqno]{amsart}
\usepackage[utf8]{inputenc}
\usepackage{amsmath} 
\usepackage{amsthm} 
\usepackage{amssymb} 
\usepackage{amscd} 
\usepackage{emptypage}
\usepackage{enumitem} 
\usepackage{mathtools} 
\usepackage{mathrsfs} 
\usepackage{hyperref} 
\hypersetup{
	linktocpage = true,
	colorlinks=true,
	linkcolor=red,
}
\usepackage{esint} 
\usepackage{dsfont} 

\usepackage{soul}

\usepackage{graphicx} 
\usepackage{subcaption}

\usepackage[colorinlistoftodos]{todonotes} 
\newtheorem{theorem}{Theorem}[section]
\newtheorem*{theorem*}{Theorem}
\newtheorem{proposition}[theorem]{Proposition}
\newtheorem{lemma}[theorem]{Lemma}
\newtheorem{corollary}[theorem]{Corollary}

\theoremstyle{definition}
\newtheorem{definition}[theorem]{Definition}

\theoremstyle{remark}
\newtheorem{remark}[theorem]{Remark}
\newtheorem*{remark*}{Remark}

\newcommand{\con}[1]{\mathbb{#1}}
\newcommand{\R}{\con{R}} 

\usepackage{euscript}
\usepackage{relsize}
\DeclareMathAlphabet{\mathpzc}{OT1}{pzc}{m}{it}
\DeclareMathAlphabet\euscr{T1}{qzc}{m}{n}

\makeatletter
\newcommand{\leqnomode}{\tagsleft@true\let\veqno\@@leqno}
\newcommand{\reqnomode}{\tagsleft@false\let\veqno\@@eqno}
\makeatother

\newcommand{\innerprod}[2]{#1 \cdot #2}

\newcommand{\s}{s}
\newcommand{\fraclaplace}{(-\Delta)^\s}
\newcommand{\fraclaplacian}{(-\Delta)^\s}

\renewcommand{\d}{\,\mathrm{d}} 

\newcommand\beqc[1]{\left\{\begin{array}{#1}}
	\newcommand\eeqc{\end{array} \right.}
\def\PDEsystem{rcll}
\def\bmatrix{\begin{pmatrix}}
	\def\ematrix{\end{pmatrix}}

\DeclareMathOperator{\sign}{sign}

\DeclareMathOperator{\PV}{P.\! V.\! }

\let\div\relax
\DeclareMathOperator{\div}{div}


\renewcommand{\d}{\,\mathrm{d}} 

\def\PDEsystem{rcll}
\def\bmatrix{\begin{pmatrix}}
	\def\ematrix{\end{pmatrix}}
\let\div\relax
\DeclareMathOperator{\div}{div}

\newcommand{\abs}[1]{\left\lvert#1\right\rvert}

\renewcommand{\div}{\textnormal{div}}

\newcommand{\suchthat}{\,:\,}

\renewcommand{\d}{\mathop{}\!\mathrm{d}}

\numberwithin{equation}{section}

\graphicspath {{Images/}}

\title[A Weierstrass extremal field theory for the fractional Laplacian]{A Weierstrass extremal field theory for the fractional Laplacian}
\author{Xavier Cabr\'e}
\address{X. Cabr\'e \textsuperscript{1,2,3}
	\newline
	\textsuperscript{1} ICREA, Pg. Lluis Companys 23, 08010 Barcelona, Spain
	\newline
	\textsuperscript{2} Universitat Polit\`ecnica de Catalunya, Departament de Matem\`{a}tiques and IMTech, 
	Diagonal 647, 08028 Barcelona, Spain
	\newline
	\textsuperscript{3} Centre de Recerca Matem\`{a}tica, Edifici C, Campus Bellaterra, 08193 Bellaterra, Spain}
\email{xavier.cabre@upc.edu}
\author{I\~{n}igo U. Erneta}
\address{I.U. Erneta \textsuperscript{2,3},
	\newline
	\textsuperscript{2} 
	Universitat Polit\`{e}cnica de Catalunya, Departament de Matem\`{a}tiques, Diagonal 647, 08028 Barcelona, Spain
	\newline
\textsuperscript{3} 
Centre de Recerca Matem\`{a}tica, Edifici C, Campus Bellaterra, 08193 Bellaterra, Spain}
\email{inigo.urtiaga@upc.edu}
\author{Juan-Carlos Felipe-Navarro} 
\address{J.C. Felipe-Navarro \textsuperscript{4}
\newline
\textsuperscript{4}
 University of Helsinki, Department of Mathematics and Statistics, Pietari Kalmin katu 5, 00014 Helsinki, Finland} 
\email{juan-carlos.felipe-navarro@helsinki.fi}
\thanks{
The three authors are supported by grants PID2021-123903NB-I00 and RED2018-102650-T funded by MCIN/AEI/10.13039/501100011033 and by ``ERDF A way of making Europe''.
The second author has received founding from the MINECO grant MDM-2014-0445-18-1.
The third author has been supported by the Academy of Finland and the European Research Council (ERC) under the European Union's Horizon 2020 research and innovation program (grant agreement No 818437).
This work is also supported by the Spanish State Research Agency, through the Severo Ochoa and Mar\'{i}a de Maeztu Program for Centers and Units of Excellence in R\&D (CEX2020-001084-M)}


\newcommand{\domain}{Q(\Omega)}
\renewcommand{\aa}{a}
\newcommand{\bb}{b}

\newcommand{\energygen}{\mathcal{E}}
\newcommand{\calibgen}{\mathcal{C}}
\newcommand{\admissiblegen}{\mathcal{A}}

\newcommand{\energy}{\mathcal{E}_{\rm N}}

\newcommand{\lag}{G_{\rm N}}

\newcommand{\region}{\mathcal{G}}

\newcommand{\energyloc}{\mathcal{E}_{\rm L}}
\newcommand{\calibloc}{\mathcal{C}_{\rm L}}
\newcommand{\lagloc}{G_{\rm L}}
\newcommand{\oploc}{\mathcal{L}_{\rm L}}
\newcommand{\neumannloc}{\mathcal{N}_{\rm L}}
\newcommand{\regionloc}{\mathcal{G}_{\rm L}}
\newcommand{\admissibleloc}{\mathcal{A}_{\rm L}}

\newcommand{\energyfrac}{\mathcal{E}_{s, F}}
\newcommand{\gagliardo}{\mathcal{E}_s}
\newcommand{\gagliardoeps}{\mathcal{E}_{K_\varepsilon}}

\newcommand{\energyquad}{\mathcal{E}_{1,F}}
\newcommand{\calibfrac}{\mathcal{C}_{s, F}}

\newcommand{\calibgags}{\mathcal{C}_{s}}
\newcommand{\calibgageps}{\mathcal{C}_{s}^{\varepsilon}}
\newcommand{\calibquad}{\mathcal{C}_{1, F}}
\newcommand{\regionfrac}{\mathcal{G}}

\newcommand{\admissiblefrac}{\mathcal{A}_{s}}

\newcommand{\energyp}{\mathcal{E}_{p\text{-Dir}}}
\newcommand{\calibp}{\mathcal{C}_{p\text{-Dir}}}

\newcommand{\calibperloc}{\mathcal{C}_{\mathcal{P}_{\rm L}}}
\newcommand{\perim}{\mathcal{P}_{\rm N}}
\newcommand{\calibper}{\mathcal{C}_{\mathcal{P}_{\rm N}}}
\newcommand{\perimloc}{\mathcal{P}_{\rm L}}

\newcommand{\normalgraph}{\nu_{\Gamma_{w}}}
\newcommand{\excess}{{\rm E}}

\newcommand{\lsn}{L^{1}_s(\R^n)}

\newcommand{\hgag}{\dot{H}^s(\Omega)}
\newcommand{\hbdy}{\mathcal{A}_{s, t_0}}

\begin{document}
\begin{abstract}
In this paper we extend, for the first time, part of the Weierstrass extremal field theory in the Calculus of Variations to a nonlocal framework.
Our model case is the energy functional for the fractional Laplacian (the Gagliardo-Sobolev seminorm), for which such a theory was still unknown.

We build a null-Lagrangian and a calibration for nonlinear equations involving the fractional Laplacian in the presence of a field of extremals.  Thus, our construction assumes the existence of a family of solutions to the Euler-Lagrange equation whose graphs produce a foliation. Then, the minimality of each leaf in the foliation follows from the existence of the calibration.
As an application, we show that monotone solutions to fractional semilinear equations are minimizers.

In a forthcoming work we generalize the theory to a wide class of nonlocal elliptic functionals and give an application to the viscosity theory.
\end{abstract}
\maketitle


\section{Introduction}

The Weierstrass extremal field theory, a classical tool from the Calculus of Variations, provides 
a sufficient condition for the minimality of critical points.
Namely, if an extremal of an elliptic functional can be embedded in a family of critical points whose graphs produce a foliation (in particular, the graphs do not intersect each other), then the given extremal is a minimizer.
The proof of this result is based on the construction of a \emph{calibration}, that is, an auxiliary functional satisfying certain properties (see Definition~\ref{def:calib}).
This theory has found important applications in the context of minimal surfaces, among others.

The purpose of this paper is to extend the classical Weierstrass field theory to the setting of functionals associated to nonlocal equations, starting here with the simplest one.
Our main result is the construction of a calibration for the fractional functional
\[ \energyfrac(w) := \dfrac{c_{n,s}}{4}\iint_{\domain} \dfrac{|w(x)-w(y)|^2}{|x-y|^{n+2s}} \d x \d y - \int_{\Omega} F(w(x)) \d x,\]
where $s \in (0,1)$, $c_{n,s}$ is a positive normalizing constant, $F \in C^1(\R),$
\begin{equation} \label{defQOmega}
	\domain := 
	(\R^{n}\times \R^{n})\setminus(\Omega^c \times \Omega^c) =
	(\Omega \times \Omega)\cup
	(\Omega \times \Omega^c) \cup (\Omega^c \times \Omega),
\end{equation}
and $\Omega \subset \R^n$ is a given bounded domain. Here and throughout the paper, $\Omega^c = \R^n\setminus\Omega$.

The Euler-Lagrange equation for the functional~$\energyfrac$ is the semilinear equation 
\[ (-\Delta)^s u = F'(u) \quad \text{ in } \Omega,\]
where
\[ (-\Delta)^s u \, (x) = c_{n,s} \PV \int_{\R^{n}} \dfrac{u(x)-u(y)}{|x-y|^{n+2s}} \d y\]
is the fractional Laplacian and $\PV$ stands for the principal value.

Our construction does not use the Caffarelli-Silvestre extension problem for the fractional Laplacian.
This is relevant, since in a subsequent paper~\cite{CabreErnetaFelipe-Calibration2} it allows us to treat more general nonlocal functionals of the form\footnote{The subindices L and N will be used throughout the work to denote local and nonlocal objects, respectively.}
\begin{equation*}
\energy(w) := \dfrac{1}{2}\iint_{\domain} \lag(x, y, w(x), w(y)) \d x \d y,
\end{equation*}
where the Lagrangian $\lag(x, y, \aa, \bb)$ is required to satisfy the natural ellipticity condition 
\begin{equation}\label{ellipticity}
	\partial^2_{\aa \bb} \lag(x, y, \aa, \bb) + \partial^2_{\aa \bb} \lag(y, x, \bb, \aa) \leq 0.
\end{equation}

As in the classical local theory, our calibration is built in the presence of a field of extremals, namely, a one-parameter family of critical points of $\energyfrac$ (or of $\energy$) whose graphs form a foliation (see Definition~\ref{def:foliation}). 
In particular, the graphs do not intersect each other.
For the construction, it suffices to have subsolutions, respectively supersolutions, on each respective side of a given extremal ---something very useful for some applications.

As a first application of our
calibration, we establish that monotone solutions to translation invariant nonlocal equations are minimizers.
This is related to a celebrated conjecture of De Giorgi for the Allen-Cahn equation.
More precisely, 
if $u$ is a solution satisfying $\partial_{x_n } u > 0$
in $\R^{n}$,
then it is a minimizer\footnote{
Monotone solutions are easily seen to be strictly stable solutions and, as a result, to be minimizers with respect to small compactly supported perturbations. Our result gives a more precise neighborhood in which the solution is minimizing.} among functions $w$ such that
\[ 
\lim_{\tau \to -\infty} u(x', \tau) \leq w(x', x_n) \leq \lim_{\tau \to +\infty} u(x', \tau)
\]
for $x = (x', x_n) \in \R^{n-1}\times \R$.
This result was only known for those nonlocal functionals for which an existence and regularity theory of minimizers is available.
We elaborate on this further in Subsection~\ref{subsection:applications}.

As a second application, in a forthcoming paper~\cite{CabreErnetaFelipe-Calibration2} we establish that minimizers of nonlocal elliptic functionals are viscosity solutions.
Although this was previously known for problems where a weak comparison principle is available (see \cite{ServadeiValdinoci, KorvenpaaKuusiLindgren, BarriosMedina}), we can prove it in more general scenarios by using the calibration technique.
This has been motivated by the theory of nonlocal minimal surfaces, where the calibration argument of the first author~\cite{Cabre-Calibration} greatly simplified the original proof (that minimizers are viscosity solutions) from~\cite{CaffarelliRoquejoffreSavin}.

\subsection{The notion of calibration} \label{subsection:calibration}

A fundamental problem in the Calculus of Variations consists of finding conditions for a function to be a minimizer of a given functional.
More precisely, given a functional $\energygen\colon \admissiblegen \to \R$ defined on some set of admissible functions $\admissiblegen$, and given $u \in \admissiblegen$, one wishes to know whether $u$ minimizes $\energygen$ among competitors in~$\admissiblegen$ having the same Dirichlet condition as $u$.

In classical local problems, the Dirichlet condition refers to the value of $u$ on the boundary of the domain $\Omega,$ while in nonlocal problems one prescribes the value in all the exterior of $\Omega,$ namely, in $\Omega^c = \R^n \setminus \Omega.$

One effective strategy to establish the minimality of a given function $u\in \admissiblegen$ consists of constructing a calibration:
\begin{definition}
	\label{def:calib}
	A functional $\calibgen\colon \admissiblegen \to \R$ is a \emph{calibration} for the functional $\energygen$ and the 
admissible function
$u \in \admissiblegen$ if the following conditions hold:
	\begin{enumerate}[label= ($\mathcal{C}$\arabic*)]
		\item \label{def:calib:2} $\calibgen(u) = \energygen(u)$.
		\item \label{def:calib:3} $\calibgen(w) \leq \energygen(w)$ for all $w \in \admissiblegen$ with the same Dirichlet condition as $u$.
		\item \label{def:calib:1} $\calibgen(w) = \calibgen(\widetilde{w})$ for all $w, \widetilde{w} \in \admissiblegen$ with the same Dirichlet condition as $u$.
	\end{enumerate}
	Functionals satisfying~\ref{def:calib:1} are known as \emph{null-Lagrangians} 
	(see, for instance, \cite[Chapter~8]{Evans} and \cite[Section~1.4]{GiaquintaHildebrandt}).
	It is, however, convenient to relax this last condition to the less stringent
	\begin{enumerate}[label= ($\mathcal{C}$\arabic*$'$)]
	\setcounter{enumi}{2}	
	\item \label{def:calib:1:prime} $\calibgen(u) \leq \calibgen(w)$ for all $w \in \admissiblegen$ with the same Dirichlet condition as $u$.	
	\end{enumerate}
	In this work we still refer to functionals satisfying~\ref{def:calib:2},~\ref{def:calib:3}, and~\ref{def:calib:1:prime} as calibrations.\footnote{In the literature, functionals satisfying~\ref{def:calib:2},~\ref{def:calib:3}, and~\ref{def:calib:1:prime} are sometimes called \emph{subcalibrations}.}
\end{definition}

Recall the meaning of the Dirichlet condition for local and nonlocal problems given right before Definition~\ref{def:calib}.

Once a calibration is available, the minimality of $u$ follows immediately both in the local and nonlocal cases.
Indeed, if $\calibgen$ is a calibration for $\energygen$ and $u \in \admissiblegen,$ then, for every $w\in\admissiblegen$ with the same Dirichlet condition as $u$, applying~\ref{def:calib:2},~\ref{def:calib:1:prime}, and~\ref{def:calib:3} (in this order) we obtain
\[
\energygen(u) = \calibgen(u) \leq \calibgen(w) \leq \energygen(w).
\]
Therefore, $u$ is a minimizer. 

\subsection{The classical theory of calibrations}

Calibrations arose in the development of the classical theory of the Calculus of Variations.
Historically, a fundamental question was to determine necessary and sufficient conditions for a function to be a minimizer.
A satisfactory answer has been obtained for 
functionals ---that we often call ``energies'', following PDE terminology--- of the form 
\begin{equation}
	\label{Intro_loc:energy}
	\energyloc(w) := \int_{\Omega} \lagloc \big(x, w(x), \nabla w(x)\big) \d x.
\end{equation}
In this framework, the function $\lagloc(x, \lambda, q)$ is called the Lagrangian of $\energyloc$.

A necessary condition for minimality is the vanishing of the first variation of $\energyloc$.
That is, every minimizer is a critical point of $\energyloc$ (an \emph{extremal}) and must satisfy the associated Euler-Lagrange equation.
If the Lagrangian $\lagloc(x, \lambda, q)$ is convex in the variables $(\lambda, q)$, then the functional $\energyloc$ is convex and, in this case, every extremal is a minimizer.
Although many models from Physics exhibit such a convexity property, it is well known that very relevant nonconvex energies appear in applications.
This is the case of the Allen-Cahn energy, among many others.
For such energy functionals, the Dirichlet problem may admit several extremals, not all of them being minimizers.
Still, if $\lagloc$ is not convex in $(\lambda, q)$, one often has that the Lagrangian $\lagloc(x, \lambda, q)$ is convex with respect to the gradient variable $q$, which amounts to the \emph{ellipticity} of the problem.

For nonconvex elliptic problems, one is interested in having sufficient conditions for an extremal to be a minimizer.
Two of such conditions, due to Jacobi and Weierstrass, are well known.
First, if a solution is strictly stable,\footnote{A solution $u$ is said to be \emph{strictly stable} if the principal eigenvalue of the linearized equation at $u$ is positive.} 
then it is a \emph{local minimizer} in a certain topology, that is, a minimizer in a small neighborhood.
This is Jacobi's condition.
The real difficulty is proving minimality in a larger, more interesting class of competitors (or perhaps even \emph{absolute minimality}).
In this direction, the Weierstrass sufficient condition yields minimality among functions taking values in a precise region.
To further elaborate on this, we need to introduce the notion of \emph{field}.
Essentially, this is a collection of ordered functions $u^t:\overline{\Omega}\to \R$, with $t$ in some interval $I\subset \R$, enjoying some regularity for the joint function $(x, t) \mapsto u^{t}(x)$.
The key point is that the graphs of these functions produce a foliation of a certain region $\region$ in $\R^n \times \R$, which allows to carry out a subtle convexity argument to bound the nonconvex functional by below with a calibration.

While fields are a classical concept in local problems, we can extend their definition to include both the local and nonlocal settings, as follows.

\begin{definition}\label{def:foliation}
Given a domain $D \subset \R^n$ (not necessarily bounded) and an interval~$I \subset \R$ (not necessarily bounded, nor open),
we say that a family $\{u^{t}\}_{t \in I}$ of functions $u^{t} \colon \overline{D} \to \R$ is a \emph{field in $D$} if
\begin{itemize}
\item the function $(x, t) \mapsto u^t(x)$ is continuous in $\overline{D} \times I$;
\item for each $x \in \overline{D}$, the function $t \mapsto u^t(x)$ is $C^1$ and increasing in $I$. 
\end{itemize}
We say that $\{u^{t}\}_{t \in I}$ is a \emph{$C^2$ field in $D$} if, additionally, the function $(x, t) \mapsto u^{t}(x)$ is $C^2$ in $\overline{D}\times I$.
\end{definition}

Given a functional $\energygen$ acting on functions defined in $\overline{D}$, and given a subdomain $\Omega\subset D$, we say that $\{u^{t}\}_{t \in I}$ is a 
\emph{field of extremals}\footnote{The term \emph{extremal field} is also often used in the literature, but we find it  ambiguous.} 
in $\Omega$ (roughly speaking, since we should refer to  $\energygen$, $D$, and $\Omega$) when it is a field in $D$ and each of the functions $u^t$ is a critical point of $\energygen$ in $\Omega$.

In the local setting, we will take $D = \Omega$. For nonlocal Lagrangians we will set $D = \R^n$ and $\Omega \subset \R^n$.

Given a field in $D$ as above, the region
\[
\region = \{ (x, \lambda) \in \overline{D}\times \R \colon \lambda = u^t(x) \text{ for some } t \in I \} \subset \R^n \times \R
\]
is foliated by the graphs of the functions $u^t$, which do not intersect each other (since $u^t(x)$ is increasing in $t$). 
In particular, we can uniquely define a \emph{leaf-parameter function}
\begin{equation}
\label{Eq:DefLeafParameter}
t \colon \mathcal{G} \to I, \quad (x, \lambda) \mapsto t(x, \lambda) \quad \text{ determined by } \quad u^{t(x,\lambda)}(x) = \lambda.
\end{equation}
The function $t$ is continuous in $\region$ by the assumptions in Definition~\ref{def:foliation}.
We will often refer to the functions $u^t$ (or their graphs) as the ``{\it leaves}'' of the field.

Having defined what a field is,
we can now state the classical theorem of Weierstrass, which was first proven for scalar ODEs:
\begin{equation}\label{thm:w}
\parbox{.85\textwidth}{
\textbf{Weierstrass sufficient condition.}
{\it
For the functional $\energyloc$ in \eqref{Intro_loc:energy}, 
assuming ellipticity (i.e., that the Lagrangian $\lagloc(x, \lambda, q)$ is convex in the gradient variable $q$),
if a critical point is embedded in a field of extremals,\footnotemark \ then the critical point is a minimizer
among functions taking values in the foliated region $\region$ and having the same boundary values as the given critical point.}
}
\end{equation}
\footnotetext{\label{footnote:sub:super}In fact, it suffices that the leaves of the field above and below the graph of the given critical point are, respectively, super and subsolutions to the Euler-Lagrange equation. This, which is well known, will be easily seen within the proofs of our main results.} 

The proof of \eqref{thm:w} is based on the construction of a calibration.
For this, given an interval $I \subset \R$, a bounded domain $\Omega \subset \R^n$, and a $C^2$ field of extremals $\{u^t\}_{t \in I}$ in $\Omega$, one considers the set of admissible functions $\admissibleloc = \{w \in C^{1}(\overline{\Omega})\suchthat {\rm graph }\, w \subset \region\}.$
Then, the functional $\calibloc \colon \admissibleloc \to \R$, defined through the Legendre transform of the Lagrangian $\lagloc$ as 
\begin{equation}\label{Intro_calib:loc}
	\begin{split}
		\calibloc(w) := & \int_{\Omega} \!\Big\{\innerprod{\partial_q \lagloc(x, u^{t}(x), \nabla u^{t}(x))}{\big(\nabla w(x) -\nabla u^{t}(x)\big)}\Big\}  \Big|_{t = t(x, w(x))} \d x\\
		& \hspace{1.5cm} + \int_{\Omega} \lagloc(x, u^{t}(x), \nabla u^{t}(x))\big|_{t = t(x, w(x))} \d x,
	\end{split}
\end{equation}
is a calibration for $\energyloc$ and each critical point $u^{t_0}, t_0\in I$.
While condition~\ref{def:calib:2} in Definition~\ref{def:calib} follows directly from~\eqref{Intro_calib:loc}, and condition~\ref{def:calib:3} amounts to ellipticity, it is a remarkable fact that the null-Lagrangian property~\ref{def:calib:1} holds. Its proof will be recalled in Section~\ref{section:local}.

As an illustrative example, in the presence of a field of extremals, the functional\footnote{The subindex $1$ in the definition of the energy functional refers to the fractional parameter $s$ in~$\energyfrac,$ which are the nonlocal analogues of $\energyquad$ treated later.
As $s$ tends to $1,$ one recovers $\energyquad$ from $\energyfrac$ after a suitable normalization; see, for instance,~\cite{DiNezzaPalatucciValdinoci-HitchhikerGuide}.}
\[ 
\energyquad(w) = \dfrac{1}{2} \int_{\Omega} |\nabla w(x)|^2 \d x - \int_{\Omega} F(w(x)) \d x 
\]
(which is typically nonconvex), admits a calibration. It is given by
\begin{equation}\label{calib:quadratic}
\calibquad(w) 
= \int_{\Omega} \!\Big\{\nabla u^t(x) \cdot (\nabla w(x) - \nabla u^{t}(x)) + \dfrac{1}{2} |\nabla u^t(x)|^2 \Big \}\Big|_{ t = t(x, w(x))}\!\! \d x - \int_{\Omega} \! F(w(x)) \d x.
\end{equation}
Notice that, although the functional $\energyquad$ is not convex, its Lagrangian is elliptic (i.e., convex in the gradient variable $q=\nabla w(x)$).

The Weierstrass sufficient condition \eqref{thm:w} naturally leads to the question of when it is possible to embed a solution of the Euler-Lagrange equation into a field of extremals in a large portion of space.
An important case corresponds to functionals which are invariant with respect to some translations.
A first example are those Lagrangians which do not depend on a direction of space and, at the same time, the extremal is monotone in that same direction.
Such a solution can be translated along the invariant direction to produce a field of extremals.
This applies to layer solutions of the Allen-Cahn equation; see Subsection~\ref{subsection:applications} below.
A second example are those Lagrangians $G_L(x,\lambda,q)$ which do not depend on the function variable $\lambda$.
In this case, a field can be obtained by translating the solution in the vertical direction.
This can be used, for instance, to show that minimal graphs are minimizing minimal surfaces;
see Section~\ref{section:nonlocal:perimeter}.

Fields of extremals can also be built in the presence of a concrete explicit solution.
Here, using the precise solution and PDE at hand, one may be able to construct a field in a more or less explicit way.
This approach has been applied in the theory of minimal surfaces to establish the minimality of Simons and Lawson cones, as explained also in Section~\ref{section:nonlocal:perimeter}.

\subsection{Nonlocal calibrations}

While the theory of calibrations for local equations is well understood, there are only two papers, to the best of our knowledge, dealing with nonlocal ones. 
In~\cite{Cabre-Calibration} the first author found an explicit calibration for the fractional perimeter, as explained in Section~\ref{section:nonlocal:perimeter}.
Pagliari~\cite{Pagliari} investigated the abstract structure of calibrations for the fractional total variation.\footnote{This functional involves the fractional perimeter of each sublevel set of a given function. The author succeeded in constructing a calibration to prove that halfspaces are minimizers, but other fields of extremals are not mentioned in this work.}

We now present our main result, which builds a calibration for the functional 
\begin{equation}\label{def:frac:energy}
\energyfrac(w) = \dfrac{c_{n,s}}{4}\iint_{\domain} \dfrac{|w(x)-w(y)|^2}{|x-y|^{n+2s}} \d x \d y - \int_{\Omega} F(w(x)) \d x
\end{equation}
in the presence of a field of extremals.
Recall \eqref{defQOmega} for the meaning of $Q(\Omega)$, Definition~\ref{def:foliation} for the notion of field, and~\eqref{Eq:DefLeafParameter} for the leaf-parameter function $t$.
The calibration properties~\ref{def:calib:2},~\ref{def:calib:3}, \ref{def:calib:1}, and~\ref{def:calib:1:prime} have been introduced in Definition~\ref{def:calib}.
As we will explain in Section~\ref{section:fractional:laplacian}, the regularity assumptions on the field can be significantly weakened.
\begin{theorem}\label{thm:frac:calibration}
Let $I \subset \R$ be an interval, $\Omega \subset \R^n$ a bounded domain, and $s \in (0,1)$.
Let $\{u^{t}\}_{t \in I}$ be a $C^2$ field in $\R^n$ in the sense of Definition \ref{def:foliation}
satisfying
\[
|u^{t}(x)| + |\partial_t u^{t}(x)| \leq C \quad \text{ for all } x \in \R^n \text{ and } t \in I,
\]
for some constant $C$.
Consider the admissible functions
\[
\admissiblefrac =
 \{ w \in 
 C^0(\R^n) \cap L^{\infty}(\R^n) 
 \suchthat  \, {\rm graph }\, w \subset \regionfrac \},
\]
where
\[
\regionfrac = \{ (x, \lambda) \in \R^n \times \R \colon \lambda = u^t(x) \text{ for some } t \in I\}.
\]

Given $t_0 \in I$ and $F \in C^1(\R)$, let 	$\calibfrac$ be the functional
\begin{align}\label{def:fraccalib}
\calibfrac(w) & :=  c_{n,s} \PV \iint_{\domain} \int_{u^{t_0}(x)}^{w(x)} \dfrac{u^t(x)-u^t(y)}{|x-y|^{n+2s}}\bigg|_{t = t(x, \lambda)} \d \lambda \d x \d y- \int_{\Omega} F(w(x)) \d x\\
		& \hspace{3cm} \quad\,  + \dfrac{c_{n,s}}{4}\iint_{\domain} \dfrac{|u^{t_0}(x)-u^{t_0}(y)|^2}{|x-y|^{n+2s}} \d x \d y
\end{align}
defined for $w \in \admissiblefrac$, where $c_{n,s}$ is the positive constant in \eqref{def:frac:energy}.

Taking $\calibgen = \calibfrac$ and $\energygen = \energyfrac$ as in~\eqref{def:frac:energy},
we have the following:
	\begin{enumerate}[label={\rm(\alph*)}]
		\item $\calibfrac$ satisfies~{\rm\ref{def:calib:2}} and~{\rm\ref{def:calib:3}} with $u = u^{t_0}$.
		\item Assume in addition that the family $\{u^{t}\}_{t \in I}$ satisfies
		\begin{equation*}\label{frac:sup:sub}
			\begin{split}
				\fraclaplacian u^{t} - F'(u^t) \geq 0 \quad \text{ in } \Omega 
				\quad
				\text{ for } t \geq t_0,
				\\
				\fraclaplacian u^{t} - F'(u^t) \leq 0 \quad \text{ in } \Omega \quad \text{ for } t \leq t_0.
			\end{split}
		\end{equation*}
		Then, $\calibfrac$ satisfies~{\rm\ref{def:calib:1:prime}} with $u = u^{t_0}$.
		In particular, $u^{t_0}$ minimizes $\energyfrac$ among functions $w$ in $\admissiblefrac$ such that $w \equiv u^{t_0}$ in $\Omega^c.$
		
		\item Assume in addition that $\{u^t\}_{t \in I}$ is a field of extremals in $\Omega$, that is,
		a field in $\R^n$ satisfying
		\begin{equation*}
			\begin{split}
				\fraclaplacian u^{t} - F'(u^t) = 0 \quad \text{ in } \Omega 
				\quad
				\text{ for all } t \in I.
			\end{split}
		\end{equation*}
		Then, the functional $\calibfrac$ satisfies~{\rm\ref{def:calib:1}} with $u = u^{t_0}$.
		Therefore, $\calibfrac$ is a calibration for $\energyfrac$ and $u^{t_0}$.
		As a consequence, for every $t \in I$, the extremal $u^t$~minimizes $\energyfrac$ among functions $w$ in $\admissiblefrac$ such that $w \equiv u^{t}$ in $\Omega^c.$
	\end{enumerate}
\end{theorem}

The meaning of the principal value $\PV$ in the definition of the functional $\calibfrac$  will be made precise in Section~\ref{section:fractional:laplacian}; see Remark~\ref{remark:pv}.

Even when the nonlocal energy functional is as simple as $\energyfrac$ (the energy functional associated to the fractional Laplacian) the form of a calibration, if any could exist, was not known prior to the present work. 

Our first attempts at constructing a calibration for $\energyfrac$ consisted on trying to ``nonlocalize'' the expression~\eqref{calib:quadratic} for the local calibration, mainly by substituting gradients by fractional ones or double integrals of differences.
This strategy seems to lead to functionals that are not calibrations.
We comment on these attempts with more detail in Appendix~\ref{section:examples}.

A second failed approach consisted of trying to find a satisfactory calibration using the extension problem for the fractional Laplacian.
Indeed, applying the local theory in the extended space gives a calibration in terms of a certain field of extremals ``upstairs'', but it was not clear at all how to write it in terms of the given field ``downstairs'' (the reason being that the functional is too involved).
Thus, the extension has not been useful to us; see Appendix~\ref{section:extension}.

We were puzzled for a long time until we revisited the work of the first author~\cite{Cabre-Calibration}, which found a calibration for the fractional perimeter.
It was written in terms of the Euler-Lagrange and Neumann operators associated to the fractional perimeter.
We then realized that such a structure was also present, but hidden, in the classical local calibration $\calibloc$ in \eqref{Intro_calib:loc}.
More precisely, for every $t_0 \in I$, in Theorem~\ref{alt:expression:local} we will see that
\begin{equation}\label{Intro_calib:loc:2}
	\begin{split}
		\calibloc(w) &= \int_{\Omega} \int_{u^{t_0}(x)}^{w(x)} \oploc(u^{t})(x)\Big|_{t = t(x, \lambda)} \d \lambda \d x  \\
		&\quad \quad \quad \quad  + \int_{\partial \Omega} \int_{u^{t_0}(x)}^{w(x)} \neumannloc(u^{t})(x)\big|_{t = t(x,\lambda)} \d \lambda \d \mathcal{H}^{n-1}(x) + \energyloc(u^{t_0}),
	\end{split}
\end{equation}
where $\oploc$ and $\neumannloc$ are, respectively, the Euler-Lagrange and Neumann operators associated to the functional $\energyloc$ in \eqref{Intro_loc:energy}.
To the best of our knowledge, this is the first time that the local calibration has been written in this way.
From this expression, the null-Lagrangian property follows easily.\footnote{Indeed, if $\{u^t\}_{t \in I}$ is a field of extremals, then $\oploc(u^{t})\equiv 0$ and the calibration depends only on the value of $w$ on $\partial \Omega$.}
Instead, traditionally one exhibited the null-Lagrangian property of $\calibloc$ by either expressing the functional as the flux of a divergence-free vector field in $\Omega \times \R$ or by certain straightforward although opaque analytic computations; see Section~\ref{section:local}.
Neither of these approaches reveals the exact role played by the Euler-Lagrange and Neumann operators in the calibration.

Once we found \eqref{Intro_calib:loc:2} for the local case, a simple extension of this expression easily led us to the nonlocal calibration of Theorem~\ref{thm:frac:calibration}.
The key point is that each of the terms in \eqref{Intro_calib:loc:2} has a clear nonlocal counterpart.
In fact, the same procedure works for general nonlocal functionals
$\energy$ of the form 
\begin{equation*}
	\energy(w) = \dfrac{1}{2}\iint_{\domain} \lag(x, y, w(x), w(y)) \d x \d y
\end{equation*}
with Lagrangian $\lag(x, y, a, b)$
satisfying the natural ellipticity condition \eqref{ellipticity}.

\subsection{An application to monotone solutions}
\label{subsection:applications}
Our main motivation to find a calibration came from the study of \emph{monotone solutions} to the fractional Allen-Cahn equation
\[
	\fraclaplacian u = u - u^3 \quad \text{ in } \R^{n}
\]
(see~\cite{CabreSireII,CabreSolaMorales}, for instance, and~\cite{CozziPassalacqua} for more general integro-differential operators).
Note that when the operator is the classical Laplacian, these solutions are related to a famous conjecture of De Giorgi; see~\cite{CabrePoggesi} for instance.

The following is an application of our main theorem to monotone solutions of translation invariant equations.

\begin{corollary}
	\label{cor:applications}
	Given $s\in (0,1)$ and 
	$F \in C^3(\R)$,
	let $u:\R^n\to\R$ be a bounded solution of
	\begin{equation}
		\label{eq:allen:cahn}
		\fraclaplacian u = F'(u) \quad \text{ in } \R^{n}. 
	\end{equation}
	Assume that $u$ is increasing in $x_n$.
	
	Then, for each bounded domain $\Omega \subset \R^{n}$, $u$ is a minimizer of $\energyfrac$ among 
	continuous functions $w:\R^n\to\R$ satisfying
	\begin{equation}\label{limit:condition}
		\lim_{\tau \to -\infty} u(x', \tau) \leq w(x', x_n) \leq \lim_{\tau \to +\infty} u(x', \tau) \quad \text{ for all } (x', x_n) \in \R^{n-1}\times \R
	\end{equation}
	and such that $w \equiv u$ in $\Omega^c$.
\end{corollary}

Let us mention that we only assume $F \in C^3$
for simplicity, to ensure that $u$ is of class $C^2$ independently of $s$.
We could weaken the regularity assumptions, but this is not the purpose of the article.

To prove Corollary~\ref{cor:applications}, we define the one-parameter family of functions $u^{t}(x) := u(x', x_n + t)$, where $x = (x', x_n) \in \R^{n-1}\times \R$.
Thanks to the monotonicity assumption, $u_{x_n}>0$, the family $\{u^t\}_{t \in \R}$ is a field in $\R^n$ in the sense of Definition~\ref{def:foliation}.
Moreover, it is a field of extremals on account of the translation invariance of the equation~\eqref{eq:allen:cahn}.
Thus, Theorem~\ref{thm:frac:calibration} gives that, on each bounded domain $\Omega \subset \R^{n}$, $u$ is a minimizer of the energy $\energyfrac$ among the admissible functions in the statement.

Let us point out that this minimality result was already known and can be proven without using calibrations.
Nevertheless, this alternative proof (described in the next paragraph) requires an existence and regularity theorem for minimizers.
Such a result is not available for many other nonlocal equations.
Our forthcoming work~\cite{CabreErnetaFelipe-Calibration2} will provide calibrations in a general nonlocal setting, and thus the proof above will allow us to show minimality of monotone solutions for the first time.

Now, we briefly discuss the proof of Corollary~\ref{cor:applications} which does not use calibrations; for the details, see~\cite{CabrePoggesi}.
One considers a minimizer within the region given by~\eqref{limit:condition} and with the same exterior datum as the monotone solution. 
Its existence and regularity can be proved in the case of equation~\eqref{eq:allen:cahn}.
Now, the monotone solution can be translated, starting from infinity, until it touches the minimizer from one side, something that must happen in $\Omega$, not in the exterior.
The strong comparison principle then yields that the translated solution and the minimizer must coincide. 
Moreover, by the exterior condition they must be equal to the original solution.
In particular, this proves that the monotone solution is a minimizer.
Furthermore, this proof gives uniqueness of solution with the exterior data of $u$.

\subsection{Outline of the article}

For the proofs of our main theorems, the reader may skip sections \ref{section:nonlocal:perimeter} and \ref{section:local}.

In Section~\ref{section:nonlocal:perimeter} we briefly comment on the classical perimeter functional
and review the work of the first author on the calibration for the nonlocal perimeter~\cite{Cabre-Calibration}.
Section~\ref{section:local} is devoted to recalling some known facts from the classical theory of calibrations and proving the new expression~\eqref{Intro_calib:loc:2}; see Theorem~\ref{alt:expression:local}.
In Section~\ref{section:fractional:laplacian} we prove Theorem~\ref{thm:frac:calibration} under weaker assumptions on the field.

In Appendix~\ref{section:extension} we apply the local theory to the Caffarelli-Silvestre extension to give a calibration for the fractional Laplacian in the extended space.
Finally, Appendix~\ref{section:examples} enumerates other natural candidates for a calibration associated to the fractional Laplacian; some of them are shown not to be calibrations.

\section{The classical and nonlocal perimeters}
\label{section:nonlocal:perimeter}

In this section we recall different notions for the perimeter of a set.
First we introduce the classical perimeter functional and its calibration.
We will mention several results concerning fields of extremals in this setting.
Later we revisit the work of the first author~\cite{Cabre-Calibration} on the construction of a calibration for the nonlocal perimeter. 
Here we will focus on identifying the key feature that leads to the calibration properties in this nonlocal framework.
This will suggest a candidate structure to search for in local functionals, which will lead to \eqref{Intro_calib:loc:2} and then allow us to treat the fractional Laplacian case.

As mentioned in the Introduction, some relevant applications of calibrations concern the theory of minimal surfaces.
In broad terms, a minimal surface $\Sigma \subset \R^n$ is a critical point of the $(n-1)$-dimensional area functional.
Given a domain $\Omega \subset \R^{n}$, the classical perimeter of a (regular) set $F \subset \R^{n}$
inside $\Omega$ is defined by
\begin{equation}
\label{perim}
\perimloc(F) := \mathcal{H}^{n-1}(\Omega \cap \partial F),
\end{equation}
where $\mathcal{H}^{n-1}$ is the $(n-1)$-dimensional Hausdorff measure.
Here we interpret the boundary of $F$ as the surface $\Sigma = \partial F$.
The critical points $E$ of $\perimloc$ (known as \emph{minimal sets} in the literature) satisfy $H_{\rm L}[\partial E] = 0$ in $\Omega$, where $H_{\rm L}[\Sigma]$ denotes the mean curvature of~$\Sigma$.
The variation of $\perimloc$ is taken with respect to perturbations preserving the boundary datum $E \cap \partial \Omega$.

We are interested in showing that certain minimal sets $E$ minimize $\perimloc$ among sets~$F$ with the same boundary condition $F \cap \partial \Omega = E \cap \partial \Omega$.
Assume that for a family of minimal sets $\{E^t\}_{t \in \R}$, the surfaces $\partial E^t$ form a foliation of $\Omega$.
For $x \in \Omega$, we let $t(x)$ be the unique $t\in \R$ such that $x \in \partial E^t$.
Denote the outward unit normal vector to $\partial F$ by~$\nu_{\partial F}$.
Then, the perimeter functional admits as a calibration
\begin{equation}
\label{class:perim}
\calibperloc(F) 
:= \int_{\Omega \cap \partial F} \! X \cdot \nu_{\partial F} \d \mathcal{H}^{n-1}, 
\end{equation}
where
\[
X(x) := \nu_{\partial E^{t}}\big|_{t = t(x)}(x)
\]
is the vector field given by the normal vectors to the surfaces $\partial E^t$.
Notice that Definition~\ref{def:calib} can be easily modified to involve subsets of $\R^n$ instead of functions.
Then, properties \ref{def:calib:2} and \ref{def:calib:3} are easy to check directly, while the null-Lagrangian property \ref{def:calib:1} follows from the divergence theorem and the fact that $\div X = 0$.
As a consequence, each $E^t$ minimizes $\perimloc$ and, therefore, each $\partial E^t$ is a minimizing minimal surface.

This discussion leads to the question of when it is possible to embed a minimal surface in a field of extremals.
The simplest situation is when the minimal surface is a graph.
If $u \colon \Omega' \subset \R^{n-1} \to \R$ is a minimal graph, then the graphs of the translations $u^t = u + t$ give an extremal field in $\Omega = \Omega' \times \R$.
By the calibration $\calibperloc$ in \eqref{class:perim}, every minimal graph is a minimizing minimal surface.\footnote{Notice that restricting the area functional to the class of graphs yields a convex functional.
In particular, every minimal graph minimizes area in this smaller class, but it is not a priori clear if they are minimizers with respect to all surfaces. The calibration is used to prove this stronger fact.}
We point out that the functional $\calibperloc$ can also be obtained by integrating a closed differential form; see Chapter $1$ in \cite{ColdingMinicozzi}.

Another interesting situation is when the minimal surface is not a graph but has an explicit expression.
Here an extremal field can still be obtained in some cases.
For instance, this is done for the Simons cone and for the more general Lawson cones in Bombieri, De Giorgi, and Giusti \cite{BombieriDeGiorgiGiusti} and in Davini \cite{Davini}.
The strategy here consists of using the symmetries of the cone to reduce the minimal surface equation to an ODE in the plane.
It is then shown that the solutions of this ODE do not intersect each other, and thus give a foliation.
We remark that, although the cone is an explicit extremal, the extremal field itself is not explicit.
An alternative is to build fields made of sub and supersolutions, which are easier to obtain and suffice to show the minimality (see footnote~\ref{footnote:sub:super}).
Explicit examples of such fields have been found for Lawson cones, simplifying the proof of the minimality; see De Philippis and Paolini~\cite{DePhilippisPaolini} for the Simons cone and Liu~\cite{LiuLawson} for Lawson cones.
We mention that the case of minimal surfaces of codimension greater than $1$ can also be treated, where the appropriate notion of calibration involves the use of differential forms; see~\cite{Morgan-CalibrationSurvey}.

The perimeter functional $\perimloc$ in \eqref{perim} has a nonlocal analogue.
Given a nonnegative symmetric kernel $K = K(z)$, with $z \in \R^n$, the $K$-nonlocal perimeter of a set $F \subset \R^n$ inside $\Omega$ is defined by 
\[
\perim(F)
:= \dfrac{1}{2}\iint_{\domain} \abs{\mathds{1}_{F}(x)-\mathds{1}_{F}(y)}K(x-y)\d x \d y,
\]
where $Q(\Omega)$ was defined in \eqref{defQOmega}.
It is well known that the Euler-Lagrange operator associated to $\perim$ is the nonlocal mean curvature $H_K$, which is defined for $F$ at boundary points $x \in \partial F$ by
\[
H_K[F](x) := \int_{\R^{n}}  \big(\mathds{1}_{F^c}(y) - \mathds{1}_{F}(y)\big) K(x-y) \d y,
\]
as introduced in~\cite{CaffarelliRoquejoffreSavin}. 
In particular, if a (sufficiently regular) set $E$ minimizes $\perim$ with respect to sets~$F$ with the same exterior values $F\setminus \Omega = E \setminus \Omega$, then $H_K[E](x) = 0$ for~$x \in \partial E \cap \Omega.$

In~\cite{Cabre-Calibration}, the first author showed that, given a measurable function $\phi\colon \R^{n} \to \R$, the functional\footnote{Let us point out that the idea of using the sign function comes from the Legendre transform of the absolute value that appears in the fractional perimeter functional.}
\begin{equation}\label{perim:calib}
	\calibper(F) := \dfrac{1}{2}\iint_{\domain} \sign\big(\phi(x)-\phi(y)\big)\big(\mathds{1}_{F}(x)-\mathds{1}_{F}(y)\big)\,K(x-y)\d x \d y
\end{equation}
is a calibration for the nonlocal perimeter $\perim$ and each superlevel set 
\[
E^t := \{x \in \R^n \suchthat \phi(x) > t\},
\] 
assuming that these sets have zero nonlocal mean curvature.
As a consequence, each~$E^t$ is a minimizer of $\perim$ with respect to sets that coincide with $E^t$ outside $\Omega$.

Properties~\ref{def:calib:2} and~\ref{def:calib:3} are easy to check directly from expression~\eqref{perim:calib}.
However, showing the null-Lagrangian property~\ref{def:calib:1} requires an alternative expression for $\calibper$.
For this, \cite{Cabre-Calibration} wrote \eqref{perim:calib} in terms of the sets $E^t$ as follows. 
Assume for simplicity that $\phi$ is smooth and $\nabla \phi(x) \neq 0$ for all $x$.
Then the level sets are smooth surfaces
\[
\partial E^{t} = \{x \in \R^{n}\colon \phi(x) = t\},
\] 
which have zero Lebesgue measure in $\R^n$, and it can be readily checked that
\begin{equation}
\label{sign}
\sign\big(\phi(x)-\phi(y)\big) = 
\big(\mathds{1}_{(E^{t})^{c}}(y) - \mathds{1}_{E^{t}}(y)\big)\big|_{t = \phi(x)} \quad \text{for a.e. } (x, y) \in \R^{n}\times \R^{n}.
\end{equation}
By skew-symmetry of $\sign(\phi(x)-\phi(y))$,
using \eqref{sign} and splitting the integration domain into $\big(\Omega \times \R^{n}\big) \cup \big(\Omega^{c}\times \Omega\big)$,
we arrive at the alternative expression
\begin{equation}\label{perim:calib:alt}
	\begin{split}
		\calibper(F) &= \int_{\Omega \cap F} H_K[E^t](x)\big|_{t = \phi(x)} \d x 
		\\
		&\quad + \int_{F \setminus \Omega} \left\{\int_{\Omega} \big(\mathds{1}_{(E^{t})^{c}}(y) - \mathds{1}_{E^{t}}(y)\big) K(x-y) \d y \right\} \bigg|_{t = \phi(x)} \d x,
	\end{split}
\end{equation}
see~\cite{Cabre-Calibration} for details.
Thus, if for all $t$ we have $H_K[E^t] = 0$ in $\Omega$, then the quantity $\calibper(F)$ depends only on the exterior condition $F \setminus \Omega$, which makes it to be a null-Lagrangian.\footnote{As mentioned in the Introduction, to show minimality one does not actually need the full null-Lagrangian property~\ref{def:calib:1} but rather the weaker condition~\ref{def:calib:1:prime}.
For instance, to prove that the set~$E^0$ minimizes $\perim,$ from \eqref{perim:calib:alt} it can be shown that it suffices for the $E^t$ ``above'' and ``below'' $E^0$ to be super and subsolutions, respectively.
For more details see~\cite{Cabre-Calibration} and compare with Theorem~\ref{thm:frac:calibration} and Proposition~\ref{property:1:loc} in the following sections.
}

Passing from~\eqref{perim:calib} to~\eqref{perim:calib:alt} is the crucial step in~\cite{Cabre-Calibration}. 
To our knowledge, the structure of the alternative expression~\eqref{perim:calib:alt} for $\calibper$ is the only way to prove the null-Lagrangian property.
The two terms in \eqref{perim:calib:alt} bring out the true dependence of $\calibper$ on the data:\footnote{The structure in~\eqref{perim:calib:alt} also appears in the local framework, where the calibration $\calibperloc$ given by \eqref{class:perim} can be written as
\[
\calibperloc(F) = \int_{\Omega\cap F} \! H_{\rm L}[\partial E^t](x)|_{t = t(x)} \d x -\int_{\partial \Omega \cap F} \! X \cdot \nu_{\partial \Omega} \d \mathcal{H}^{n-1}.
\]
Moreover, it is not difficult to see that the calibration for the \emph{fractional perimeter}, i.e., the $K$-nonlocal perimeter with $K(z) = |z|^{-n-2s}$ for $s \in (0,1)$, recovers $\calibperloc$ in the limit when $s \to 1.$}
the first one involves the Euler-Lagrange equation of $\perim$ at each superlevel set $E^t$, while the second one depends only on the set $F$ outside $\Omega$. 
The existence of such a structure for the nonlocal perimeter suggested that it could also be present, although hidden, in other calibrations, even in the local case, as we will see in next section.


\section{The theory of calibrations for local equations, and a novelty}
\label{section:local}

The purpose of this section is twofold:
first, to review the classical theory of fields of extremals and calibrations for ``local'' functionals 
and, second, to give a new proof of the calibration properties in this setting.
Inspired by the structure of the calibration~\eqref{perim:calib:alt} for the nonlocal perimeter,
we will find an alternative expression for the classical calibration \eqref{calib:loc:2} (a new expression to the best of our knowledge), which involves only the Euler-Lagrange and Neumann operators acting on the field.

Consider an energy functional of the form
\begin{equation} \label{def:energy:loc}
	\energyloc(w) := \int_{\Omega} \lagloc \big(x, w(x), \nabla w(x)\big) \d x,
\end{equation}
where the Lagrangian $\lagloc(x, \lambda, q)$ is of class $C^{2}$ in all arguments.

A function that plays an important role when studying minimality to nonconvex energy functionals of the form~\eqref{def:energy:loc} is the so-called \emph{Weierstrass excess function}. 
It is defined 
for $x \in \Omega$, $\lambda \in\R$, and $q$, $\widetilde{q} \in \R^n$
by
\begin{equation}\label{excess:function}
	\excess(x, \lambda, q, \widetilde{q}) := \lagloc(x, \lambda, \widetilde{q}) - \lagloc(x, \lambda, q) - \partial_{q}\lagloc(x, \lambda, q)\cdot(\widetilde{q}-q).
\end{equation}
It is well known (see \cite{GiaquintaHildebrandt}) that if $u \in C^1(\overline{\Omega})$ is a minimizer of $\energyloc$ with respect to small $C^0_c(\Omega)$ perturbations,\footnote{This type of local minimizers are often referred to as \emph{strong minimizers} in the literature.} then it must satisfy the \emph{Weierstrass necessary condition}
\begin{equation}\label{weierstrass:nec}
	\excess(x, u(x), \nabla u(x), \xi) \geq 0 \quad \text{ for all } x \in \Omega, \, \xi \in \R^n.
\end{equation}
Note that condition~\eqref{weierstrass:nec} on the excess function is automatically satisfied by every function $u$ whenever $\lagloc(x, \lambda, q)$ is convex with respect to the variable $q$, i.e., when the problem is elliptic.
The Dirichlet energy and more generally the Lagrangian associated to the $p$-Laplacian are important elliptic examples where \eqref{weierstrass:nec} is thus automatically satisfied.

Given an interval $I \subset \R$ and $\{u^{t}\}_{t \in I}$ a $C^2$ field in $\Omega$ (in the sense of Definition~\ref{def:foliation}),
we let 
\[
\regionloc := \big\{(x, \lambda)\in \overline{\Omega} \times \R \suchthat \lambda = u^{t}(x) \text{ for some } t \in I \big\}
\]
and consider the set of admissible functions
\[
\admissibleloc := \big\{w \in C^{1}(\overline{\Omega})\suchthat {\rm graph }\, w \subset \regionloc\big\}.
\]
In the classical theory, one employs the Legendre transform of $G_L$ to define the functional 
$\calibloc \colon \admissibleloc \to \R$ by
\begin{equation}\label{calib:loc}
	\begin{split}
		\calibloc(w) := & \int_{\Omega} \Big\{ \innerprod{\partial_q \lagloc(x, u^{t}(x), \nabla u^{t}(x))}{\big(\nabla w(x) -\nabla u^{t}(x)\big)} \Big\} \Big|_{t = t(x, w(x))} \d x\\
		& \hspace{1.5cm} + \int_{\Omega} \lagloc(x, u^{t}(x), \nabla u^{t}(x))\big|_{t = t(x, w(x))} \d x.
	\end{split}
\end{equation}

Under the assumption that $\{u^t\}_{t \in I}$ is a field of extremals and every leaf $u^t$ satisfies the Weierstrass necessary condition~\eqref{weierstrass:nec}, it is well known that $\calibloc$ is a calibration for the functional $\energyloc$ and each $u^t$;\footnote{The positivity of the excess function $\excess$ for every leaf $u^t$ is only required to show property~\ref{def:calib:3}. Properties~\ref{def:calib:2} and \ref{def:calib:1} follow directly from the existence of the extremal field.} see~\cite{AlbertiAmbrosioCabre, AngrisaniAscioneLeoneMantegazza, GiaquintaHildebrandt}.
As an illustrative example, the $p$-Dirichlet energy
\[ 
\energyp(w) = \dfrac{1}{p}\int_{\Omega} \abs{\nabla w(x)}^{p} \d x 
\]
admits the calibration 
\[
\calibp(w) = \int_{\Omega} \Big\{\abs{\nabla u^t(x)}^{p-2} \nabla u^t(x) \cdot \big(\nabla w(x) - \nabla u^{t}(x)\big) + \dfrac{1}{p}\abs{\nabla u^t(x)}^{p}\Big\}\Big|_{t = t(x, w(x))} \d x.
\]

We will give a new proof that the functional $\calibloc$ is a calibration.
As mentioned before, the key point in our approach is to rewrite $\calibloc$
in an alternative form involving only those operators which are of interest to the theory of PDE: the Euler-Lagrange and Neumann operators.
These arise when computing the 
first variation of $\energyloc$ at $u\in C^{2}(\overline{\Omega})$ in a direction of $\eta \in C^{\infty}(\overline{\Omega})$, that is,
\begin{equation}
\begin{split}
\label{firstvar}
\dfrac{\d}{\d \varepsilon} \energyloc(w + \varepsilon \eta)\Big|_{\varepsilon = 0} = \int_{\Omega} \oploc (w)(x) \, \eta(x) \d x  + \int_{\partial\Omega} \neumannloc(w)(x) \, \eta(x)\d \mathcal{H}^{n-1}(x).
\end{split}
\end{equation}
Here in \eqref{firstvar}, $\oploc$ denotes the Euler-Lagrange operator
\begin{equation}
\label{loc:operator}
	\oploc (w)(x) := - \div\big(\partial_q \lagloc\big(x, w(x), \nabla w(x)\big)\big) + \partial_{\lambda} \lagloc\big(x, w(x), \nabla w(x)\big)
\end{equation}
and $\neumannloc$ denotes the Neumann operator
\begin{equation}
\label{loc:neumann}
	\neumannloc (w)(x) := \partial_q \lagloc\big(x, w(x), \nabla w(x)\big)\cdot \nu_{\partial \Omega}(x),
\end{equation}
where
$\nu_{\partial \Omega}$ is the outward unit normal vector to $\partial \Omega$.

The following identity is our new result.
\begin{theorem}
\label{alt:expression:local}
	Given an interval $I \subset \R$ and a bounded domain $\Omega \subset \R^n$,
	let  $\{u^t\}_{t \in I}$ be a $C^2$ field in $\Omega$ in the sense of Definition \ref{def:foliation}.
Let $\lagloc = \lagloc(x, \lambda, q)$ be a $C^2$ function.
	
Then, for any $t_0 \in I$, the functional $\calibloc$ defined in \eqref{calib:loc} can be written as
\begin{equation}\label{calib:loc:2}
	\begin{split}
		\calibloc(w) &= \int_{\Omega} \int_{u^{t_0}(x)}^{w(x)} \oploc(u^{t})(x)\Big|_{t = t(x, \lambda)} \d \lambda \d x  \\
		&\quad \quad \quad \quad  + \int_{\partial \Omega} \int_{u^{t_0}(x)}^{w(x)} \neumannloc(u^{t})(x)\big|_{t = t(x,\lambda)} \d \lambda \d \mathcal{H}^{n-1}(x) + \energyloc(u^{t_0}).
	\end{split}
\end{equation}
\end{theorem}

The proof of the theorem follows a typical
strategy for showing the null-Lagrangian property, as seen for instance in~\cite{AlbertiAmbrosioCabre}.
However, a new trick will allow us to identify in the expression the operators $\oploc$ and $\neumannloc$ acting on the leaves.
This is the first time we have seen the calibration written this way.

\begin{proof}[Proof of Theorem~\ref{alt:expression:local}]
In order to prove the result it will be enough to show that
\begin{equation}\label{calib:loc:2'}
\begin{split}
	\calibloc(w) &- \calibloc(\widetilde{w}) = \int_{\Omega} \int_{\widetilde{w}(x)}^{w(x)} \oploc(u^{t})(x)\Big|_{t = t(x, \lambda)} \d \lambda \d x  \\
	&\quad \quad \quad \quad \quad \quad \quad \quad + \int_{\partial \Omega} \int_{\widetilde{w}(x)}^{w(x)} \neumannloc(u^{t})(x)\big|_{t = t(x,\lambda)} \d \lambda \d \mathcal{H}^{n-1}(x).
\end{split}
\end{equation}
for any given $w, \widetilde{w} \in \admissibleloc$. That is, we only need to take $\tilde{w}=u^{t_0}$ and use the easy equality $\calibloc(u^{t_0}) = \energyloc(u^{t_0})$.

First, let us briefly describe the proof of identity~\eqref{calib:loc:2'}. We consider $\calibloc$ acting on the convex combination $w_{\theta} := (1-\theta) \widetilde{w} + \theta w$
and express the left-hand side of \eqref{calib:loc:2'} as $\int_{0}^{1} \frac{\d}{\d\theta} \calibloc(w_{\theta}) \d \theta$.
While in the literature the functions $w$ and $\widetilde{w}$ are assumed to have the same boundary conditions, here we do not impose such a restriction.
Then, we compute the derivative in $\theta$ using the expression of $\calibloc(w_{\theta})$ as an integral in $x$ and integrating by parts.
Finally, after applying Fubini's theorem to interchange the order of integration,
the key point is to make the change of variables $\theta \mapsto w_{\theta}(x)$ for each $x$.
This yields the final expression.
	
Next, let us proceed with the proof. We let $\zeta := w - \widetilde{w}$ and hence $w_{\theta} = \widetilde{w} + \theta \zeta$. 
Since
\begin{equation}
\label{lel0}
\begin{split}
&\calibloc(w) - \calibloc(\widetilde{w}) = \int_{0}^{1}\dfrac{\d}{\d \theta} \calibloc(w_{\theta}) \d \theta
\end{split}	
\end{equation}
and
\begin{equation}
\label{lel1}
\begin{split}
\dfrac{\d}{\d \theta} \calibloc(w_{\theta}) &= \int_{\Omega} \dfrac{\d}{\d \theta}  \Big\{ \lagloc\big(x, u^{t}, \nabla u^{t}\big) \big|_{t = t(x, w_{\theta}(x))} \Big\} \d x\\
& \quad \quad \quad  + 
\int_{\Omega} \dfrac{\d}{\d \theta}\left( \Big\{ \innerprod{\partial_q \lagloc\big(x, u^{t}, \nabla u^{t}\big)}{\big(\nabla w_{\theta} -\nabla u^{t}\big)} \Big\}\Big|_{t = t(x, w_{\theta}(x))} \right) \d x,
\end{split}
\end{equation}
we must compute each of the integrands in \eqref{lel1}.

By the chain rule, the first integrand can be written as
\[
\begin{split}
&\dfrac{\d}{\d \theta} \Big\{ \lagloc\big(x, u^{t}, \nabla u^{t}\big)\big|_{t = t(x, w_{\theta}(x))}\Big\} \\
&\hspace{3cm} =  \partial_{\lambda}\lagloc\big(x, u^{t}, \nabla u^{t}\big)\big|_{t = t(x, w_{\theta}(x))} \frac{\d}{\d \theta} \big( u^{t(x, w_{\theta}(x))} \big) \\
& \hspace{4.5cm} + \innerprod{\partial_{q}\lagloc\big(x, u^{t}, \nabla u^{t}\big)\big|_{t = t(x, w_{\theta}(x))}}{\dfrac{\d}{\d \theta}\big(\nabla u^{t}|_{t = t(x, w_{\theta}(x))}\big)}\\
\end{split}
\]
and using that $\frac{\d}{\d \theta} \big( u^{t(x, w_{\theta}(x))} (x)\big) = \frac{\d}{\d \theta} w_{\theta}(x) = \zeta(x)$, we deduce
\begin{equation}
\label{lel2}
\begin{split}
&\dfrac{\d}{\d \theta} \Big\{ \lagloc\big(x, u^{t}, \nabla u^{t}\big)\big|_{t = t(x, w_{\theta}(x))}\Big\} \\
&\ \ \ =  \partial_{\lambda}\lagloc\big(x, u^{t}, \nabla u^{t}\big)\big|_{t = t(x, w_{\theta}(x))}\zeta + \innerprod{\partial_{q}\lagloc\big(x, u^{t}, \nabla u^{t}\big)\big|_{t = t(x, w_{\theta}(x))}}{\dfrac{\d}{\d \theta}\big(\nabla u^{t}|_{t = t(x, w_{\theta}(x))}\big)}.
\end{split}
\end{equation}
Similarly, the second integrand in \eqref{lel1} is
\begin{equation}
\label{lel3}
\begin{split}
&\dfrac{\d}{\d \theta}\left( \Big\{ \innerprod{\partial_q \lagloc\big(x, u^{t}, \nabla u^{t}\big)}{\big(\nabla w_{\theta} -\nabla u^{t}\big)} \Big\}\Big|_{t = t(x, w_{\theta}(x))} \right)\\
& \quad \quad \quad  =  \Big\{ \innerprod{\partial_{t}\big[\partial_q \lagloc\big(x, u^{t}, \nabla u^{t}\big)\big]}{\big(\nabla w_{\theta} - \nabla u^{t}\big)}  \Big\}\Big|_{t = t(x, w_{\theta}(x))} \dfrac{\d}{\d \theta}\big[ t(x, w_{\theta}(x))\big]\\
&\hspace{3.5cm} +\innerprod{\partial_{q}\lagloc\big(x, u^{t}, \nabla u^{t}\big)\big|_{t = t(x, w_{\theta}(x))}}{\frac{\d}{\d \theta}\nabla w_{\theta}} 
\\
&\hspace{3.5cm} 
- \innerprod{\partial_{q}\lagloc\big(x, u^{t}, \nabla u^{t}\big)\big|_{t = t(x, w_{\theta}(x))}}{\dfrac{\d}{\d \theta} \big(\nabla u^{t}|_{t = t(x, w_{\theta}(x))}\big)}\\
& \quad \quad \quad  =  \Big\{ \innerprod{\partial_{t}\big[\partial_q \lagloc\big(x, u^{t}, \nabla u^{t}\big)\big]}{\big(\nabla w_{\theta} - \nabla u^{t}\big)}  \Big\}\Big|_{t = t(x, w_{\theta}(x))} \partial_{\lambda} t(x, w_{\theta}(x)) \zeta \\
&\hspace{3.5cm} +\innerprod{\partial_{q}\lagloc\big(x, u^{t}, \nabla u^{t}\big)\big|_{t = t(x, w_{\theta}(x))} }{\nabla \zeta} \\
&\hspace{3.5cm} - \innerprod{\partial_{q}\lagloc\big(x, u^{t}, \nabla u^{t}\big)\big|_{t = t(x, w_{\theta}(x))}}{\dfrac{\d}{\d \theta} \big(\nabla u^{t}|_{t = t(x, w_{\theta}(x))}\big)}.
\end{split}
\end{equation}

Adding \eqref{lel2} and \eqref{lel3}, substituting in \eqref{lel1}, and rearranging terms, we see that
\begin{equation}\label{calib:diff:aux:1}
\begin{split}
&\dfrac{\d}{\d \theta} \calibloc(w_{\theta}) \\
&\ \ \ = \int_{\Omega} \partial_{\lambda}\lagloc\big(x, u^{t}, \nabla u^{t}\big)\big|_{t = t(x, w_{\theta}(x))}\zeta \d x +\int_{\Omega}\innerprod{\partial_{q}\lagloc\big(x, u^{t}, \nabla u^{t}\big)\big|_{t = t(x, w_{\theta}(x))}}{\nabla \zeta} \d x\\
&\ \ \ \quad \quad \quad + \int_{\Omega}  \Big\{ \innerprod{\partial_{t}\big[\partial_q \lagloc\big(x, u^{t}, \nabla u^{t}\big)\big]}{\big(\nabla w_{\theta} - \nabla u^{t}\big)}  \Big\}\Big|_{t = t(x, w_{\theta}(x))} \partial_{\lambda}t(x, w_{\theta}(x)) \zeta \d x.
\end{split}
\end{equation}
The second term in~\eqref{calib:diff:aux:1} can be integrated by parts as
\begin{equation}\label{calib:diff:aux:2}
\begin{split}
&\int_{\Omega}\innerprod{\partial_{q}\lagloc\big(x, u^{t}, \nabla u^{t}\big) \big|_{t = t(x, w_{\theta}(x))}}{\nabla \zeta} \d x \\
&\quad \quad  = \int_{\partial\Omega} \partial_{q}\lagloc\big(x, u^{t},\nabla u^{t}\big)\big|_{t = t(x, w_{\theta}(x))} \, \cdot \nu_{\partial\Omega} \ \zeta \d \mathcal{H}^{n-1} \\
& \quad \quad \quad \quad \quad -\int_{\Omega} \div\big(\partial_q \lagloc\big(x, u^{t}, \nabla u^{t}\big)\big)\big|_{t = t(x, w_{\theta}(x))} \zeta \d x\\
& \quad \quad \quad \quad \quad  - \int_{\Omega} \innerprod{\partial_{t}\big[\partial_q \lagloc\big(x, u^{t}, \nabla u^{t}\big)\big]\big|_{t = t(x, w_{\theta}(x))} }{\nabla \big[t(x, w_{\theta}(x))\big]} \zeta \d x.
\end{split}
\end{equation}

We now claim that
\begin{equation}
\label{miracle}
\nabla \big[t(x, w_{\theta}(x))\big] = \big(\nabla w_{\theta} - \nabla u^{t}\big) \Big|_{t = t(x, w_{\theta}(x))} \partial_{\lambda}t(x, w_{\theta}(x)),
\end{equation}
which leads to the identity we wish to prove.
Indeed, if \eqref{miracle} holds, then,
substituting \eqref{calib:diff:aux:2} in \eqref{calib:diff:aux:1}, we have
\begin{equation}
\label{miracle2}
\begin{split}
\dfrac{\d}{\d \theta} \calibloc(w_{\theta}) &= \int_{\Omega} \big\{\partial_{\lambda} \lagloc\big(x, u^{t}, \nabla u^{t}\big) - \div\big(\partial_q \lagloc\big(x, u^{t}, \nabla u^{t}\big)\big) \big\}\big|_{t = t(x, w_{\theta}(x))} \zeta \d x\\
& \quad  \quad \quad  \quad  + \int_{\partial\Omega} \partial_{q}\lagloc\big(x, u^{t},\nabla u^{t}\big)\big|_{t = t(x, w_{\theta}(x))}\,\cdot\nu_{\partial\Omega}\ \zeta \d \mathcal{H}^{n-1}\\
& = \int_{\Omega}  \oploc(u^{t})\big|_{t = t(x, w_{\theta}(x))} \zeta \d x + \int_{\partial\Omega} \neumannloc(u^{t})\big|_{t = t(x, w_{\theta}(x))} \zeta \d \mathcal{H}^{n-1}.
\end{split}
\end{equation}
Thus,
using \eqref{miracle2} in \eqref{lel0}, by Fubini's theorem we deduce
\[
\begin{split}
&\calibloc(w) - \calibloc(\widetilde{w}) \\
&\hspace{1cm} = \int_{\Omega} \int_{0}^{1} \oploc(u^{t})\big|_{t = t(x, w_{\theta}(x))} \zeta(x) \d \theta \d x + \int_{\partial \Omega}\int_{0}^{1} \neumannloc(u^{t})\big|_{t = t(x, w_{\theta}(x))} \zeta(x) \d \theta \d \mathcal{H}^{n-1}\\
&\hspace{1cm} = \int_{\Omega} \int_{\widetilde{w}(x)}^{w(x)} \oploc(u^{t})\big|_{t = t(x, \lambda)} \d \lambda \d x + \int_{\partial \Omega}\int_{\widetilde{w}(x)}^{w(x)} \neumannloc(u^{t})\big|_{t = t(x,\lambda)} \d \lambda \d \mathcal{H}^{n-1},
\end{split}
\]
where in the last line we have applied the change of variables $\lambda = \widetilde{w}(x) + \theta \zeta(x)$ for each $x$.

To show the claim \eqref{miracle}, we use the definition of the leaf-parameter function $t(x, \lambda)$.
Differentiating $u^{t}(x)\big|_{t = t(x, \lambda)} = \lambda$ with respect to $\lambda$, we have
\begin{equation}
\label{miracle3}
\partial_{t} u^{t} \big|_{t = t(x, \lambda)}  \partial_{\lambda} t(x,\lambda) = 1.
\end{equation}
Moreover,
taking the gradient of the identity $u^{t(x, w_{\theta}(x))}(x) = w_{\theta}(x)$, we obtain
\begin{equation}
\label{miracle4}
\nabla u^{t}(x)\big|_{t = t(x, w_{\theta}(x))} + \partial_t u^{t}(x)\big|_{t = t(x, w_{\theta}(x))} \nabla\big[t(x, w_{\theta}(x))\big]  = \nabla w_{\theta}(x).
\end{equation}
Multiplying \eqref{miracle4} by $\partial_{\lambda} t(x, w_{\theta}(x))$, applying \eqref{miracle3} with $\lambda = w_{\theta}(x)$, and rearranging terms
leads to \eqref{miracle}
and concludes the proof.
\end{proof}

\begin{remark}
The expression~\eqref{calib:loc:2} can be 
deduced 
in a more geometric way using the divergence theorem in $\R^{n+1}$.
As we see next, this gives an alternative proof of Theorem~\ref{alt:expression:local}.
Consider the vector field 
$X \colon \overline{\Omega} \times \R 
\to \R^{n+1} = \R^{n} \times \R$ given by
\[
X(x, \lambda) = \big(X^{x}(x,\lambda), X^{\lambda}(x, \lambda)\big),
\] 
where
\[
\begin{split}
&X^{x}(x,\lambda) := -\partial_q \lagloc\big(x, u^{t}(x), \nabla u^{t}(x)\big)\big|_{t = t(x, \lambda)}, \\
&X^{\lambda}(x,\lambda) := \Big\{-\innerprod{\partial_q \lagloc\big(x, u^{t}(x), \nabla u^{t}(x)\big)}{\nabla u^{t}(x)}+ \lagloc\big(x, u^{t}(x), \nabla u^{t}(x)\big) \Big\}\Big|_{t = t(x, \lambda)}.
\end{split}
\]
Then, 
an easy computation from \cite{AlbertiAmbrosioCabre} shows
\begin{equation*}
\div X (x, \lambda) = \oploc(u^{t})(x)\big|_{t = t(x, \lambda)},
\end{equation*}
where $\div$ is the divergence in $\R^{n+1}$,
i.e., $\div X (x, \lambda) = \div_x X^{x}(x, \lambda) + \partial_{\lambda} X^{\lambda}(x,\lambda)$.
From the definition of $X$, 
it can also be checked that $\calibloc$ can be written in the compact form
\begin{equation*}
\calibloc(w) = \int_{\Gamma_{w}} \innerprod{X}{\normalgraph} \d \mathcal{H}^{n},
\end{equation*}
where $\Gamma_{w} \subset \R^{n+1}$ is the graph of $w$ and $\normalgraph$ is the unit vector normal to $\Gamma_w$ pointing ``upwards''.
In coordinates, $\normalgraph$ reads $\normalgraph(x, w(x)) = \big(1+\abs{\nabla w(x)}^2\big)^{-1/2} (-\nabla w (x), 1).$	

\begin{figure}
\centering
\definecolor{cffeeaa}{RGB}{255,238,170}
\definecolor{cffccaa}{RGB}{255,204,170}

\begin{tikzpicture}[y=0.80pt, x=0.80pt, yscale=-5.000000, xscale=5.000000, inner sep=0pt, outer sep=0pt]

    \path[color=black,fill=blue,line join=miter,line cap=butt,miter
      limit=4.00,nonzero rule,line width=0.042pt] (29.6992,30.4629-5) .. controls
      (29.1361,30.4670-5) and (27.7079,30.4705-5) .. (26.1426,30.4746-5) --
      (26.1426,30.6738-5) .. controls (27.7079,30.6697-5) and (29.1365,30.6662-5) ..
      (29.7012,30.6621-5) -- cycle;
      
    \path[draw=blue,fill=blue,even odd rule,line width=0.023pt]
      (26.9420,30.5726-5) -- (27.3410,30.1715-5) -- (25.9420,30.5752-5) --
      (27.3430,30.9715-5) -- cycle;
      
    \path[color=black,fill=blue,line join=miter,line cap=butt,miter
      limit=4.00,nonzero rule,line width=0.042pt] (87.4590,30.1367-5) --
      (87.4570,30.3359-5) .. controls (88.0216,30.3400-5) and (89.4502,30.3435-5) ..
      (91.0156,30.3477-5) -- (91.0156,30.1484-5) .. controls (89.4503,30.1443-5) and
      (88.0221,30.1408-5) .. (87.4589,30.1367-5) -- cycle;
    
    \path[draw=blue,fill=blue,even odd rule,line width=0.023pt]
      (90.2157,30.2465-5) -- (89.8147,30.6455-5) -- (91.2157,30.2491-5) --
      (89.8168,29.8455-5) -- cycle;
      
	  	\path[scale=0.265,fill=cffeeaa,miter limit=4.00,line width=0.034pt]
	    (248.4948,170.3810) .. controls (243.8108,169.8505) and (239.2693,168.9106) ..
	    (235.2019,166.6436) .. controls (231.8972,164.8018) and (229.0456,162.6045) ..
	    (225.6890,159.2371) .. controls (222.5339,156.0717) and (219.7480,152.2290) ..
	    (217.0186,147.6191) .. controls (214.6962,143.6967) and (210.2772,135.2959) ..
	    (210.2772,134.7526) .. controls (210.2772,134.3948) and (218.9600,122.4623) ..
	    (223.6965,116.1711) .. controls (235.5441,100.4349) and (243.2341,93.0184) ..
	    (250.4808,90.5728) .. controls (252.2830,89.9646) and (254.2785,89.5708) ..
	    (256.2193,89.6943) .. controls (259.7036,89.9159) and (261.8550,91.4802) ..
	    (266.0664,95.3519) .. controls (272.3290,101.1093) and (276.7479,103.1490) ..
	    (284.2365,103.4326) .. controls (287.3718,103.5514) and (290.0198,103.1876) ..
	    (292.7569,102.6087) .. controls (295.6144,102.0042) and (298.5521,101.0776) ..
	    (302.2353,99.5207) .. controls (309.9325,96.2671) and (318.7964,90.3971) ..
	    (326.8853,83.2576) -- (330.8402,79.5508) -- (330.8512,112.7104) --
	    (330.8452,144.0525) -- (329.3680,144.9070) .. controls (327.0893,146.1783) and
	    (311.9145,153.7603) .. (306.5445,156.0283) .. controls (297.0419,160.0418) and
	    (289.6908,162.8825) .. (280.6681,165.6119) .. controls (267.8495,169.4896) and
	    (256.2580,171.2604) .. (248.4943,170.3810) -- cycle;

	  \path[scale=0.265,fill=cffccaa,miter limit=4.00,line width=0.012pt]
	    (134.9339,83.9674) .. controls (132.9261,84.0283) and (121.3157,84.5578) ..
	    (112.0391,85.6807) -- (112.3243,147.6909) .. controls (119.3143,155.6125) and
	    (126.4392,163.0363) .. (134.3733,169.4226) .. controls (140.1664,173.5658) and
	    (144.9706,177.3578) .. (151.6438,179.3610) .. controls (158.2150,181.2590) and
	    (164.8303,179.8752) .. (170.1107,177.2252) .. controls (178.3702,173.1171) and
	    (184.9008,166.4395) .. (191.0010,159.6507) .. controls (197.9067,151.8528) and
	    (204.3319,143.2860) .. (210.4058,134.8355) .. controls (209.6376,132.8856) and
	    (207.8880,129.2023) .. (206.6377,126.6526) .. controls (202.9059,118.9554) and
	    (197.6027,107.1287) .. (189.8615,99.0350) .. controls (184.0529,92.7388) and
	    (178.4376,89.0628) .. (171.9858,87.2778) .. controls (162.9049,84.6701) and
	    (153.2948,84.1772) .. (143.8922,84.0120) .. controls (142.8217,83.9990) and
	    (139.5621,83.9188) .. (138.4915,83.9232) .. controls (138.4915,83.9232) and
	    (136.9416,83.9065) .. (134.9339,83.9674) -- cycle;
    
  \path[draw=black,line join=miter,line cap=butt,miter limit=4.00,line
    width=0.868pt] (29.6669,22.6812) .. controls (29.6669,22.6812) and
    (42.0580,21.0726) .. (47.2550,23.8120) .. controls (55.0619,27.9273) and
    (55.2098,41.8742) .. (63.5872,44.6496) .. controls (71.4009,47.2383) and
    (87.4115,38.1538) .. (87.4115,38.1538);

    \path[color=red,fill=red,line join=miter,line cap=butt,miter
      limit=4.00,nonzero rule,line width=0.042pt] (41.6738,44.1641) .. controls
      (41.6697,45.7294) and (41.6662,47.1576) .. (41.6621,47.7207) --
      (41.8613,47.7227) .. controls (41.8654,47.1581) and (41.8709,45.7294) ..
      (41.8750,44.1641) -- cycle;

    \path[draw=red,fill=red,even odd rule,line width=0.023pt]
      (41.7720,44.9643) -- (42.1710,45.3654) -- (41.7746,43.9643) --
      (41.3710,45.3633) -- cycle;
      
  \path[draw=red,line join=miter,line cap=round,miter limit=4.00,line
    width=0.868pt] (29.7191,39.0766) .. controls (29.7191,39.0766) and
    (36.6360,47.4972) .. (41.4991,47.7060) .. controls (52.6552,48.1848) and
    (62.5297,17.0220) .. (70.2537,25.0861) .. controls (76.8843,32.0085) and
    (87.4889,21.1137) .. (87.4889,21.1137);

    \path[color=black,fill=black,line join=miter,line cap=butt,miter
      limit=4.00,nonzero rule,line width=0.042pt] (46.3770,20.1270) --
      (45.0566,22.9609) -- (45.2383,23.0449) -- (46.5586,20.2129) -- cycle;

    \path[draw=black,fill=black,even odd rule,line width=0.023pt] (46.1295,20.8949)
      -- (46.3231,21.4265) -- (46.5521,19.9886) -- (45.5980,21.0885) -- cycle;

    \path[color=black,fill=black,line join=miter,line cap=butt,miter
      limit=4.00,nonzero rule,line width=0.042pt] (75.5996,40.2520) --
      (75.4121,40.3203) -- (76.4355,43.1230) -- (76.6230,43.0547) -- cycle;
      
    \path[draw=black,fill=black,even odd rule,line width=0.023pt] (75.7802,41.0374)
      -- (76.2932,41.2760) -- (75.4373,40.0981) -- (75.5417,41.5503) -- cycle;

    \path[color=black,fill=red,line join=miter,line cap=butt,miter
      limit=4.00,nonzero rule,line width=0.042pt] (72.0117,22.4590) --
      (70.0000,24.8477) -- (70.1543,24.9766) -- (72.1641,22.5879) -- cycle;

    \path[draw=red,fill=red,even odd rule,line width=0.023pt]
      (71.5729,23.1355) -- (71.6213,23.6991) -- (72.2169,22.3704) --
      (71.0093,23.1839) -- cycle;

      \path[draw=black,line join=miter,line cap=butt,line width=0.856pt]
      (29.6711,11.2500) -- (29.6711,58.1042) -- (87.5015,58.1042) --
      (87.5015,11.2500);
      
      \path[draw=blue,line join=miter,line cap=butt,line width=0.856pt]
      (29.6711,22.6812) -- (29.6711,39.0766);
      
      \path[draw=blue,line join=miter,line cap=butt,line width=0.856pt]
      (87.5015,21.1137) -- (87.5015,38.1538);

  \path[fill=blue,line width=0.056pt] (40.2406,32.9619) node[above right]
    (text13332) {\color{orange} $R^+$};
    
  \path[fill=blue,line width=0.056pt] (67.4960,32.9588) node[above right]
    (text13332-5) {\color{orange} $R^-$};
    
  \path[fill=black,line width=0.056pt] (57.5607,61.5095) node[above right]
    (text13332-6) {$\Omega$};
    
  \path[fill=black,line width=0.056pt] (47.5213,21.6198) node[above right]
    (text13332-6-9) {\color{black} $\nu_{\Gamma_{w}}$};
    
  \path[fill=red,line width=0.056pt] (42.5213,45.0198) node[above right]
    (text13332-6-9-9) {\color{red} $\nu_{\Gamma_{\widetilde{w}}}$};
    
  \path[fill=red,line width=0.056pt] (72.7213,25.0198) node[above right]
    (text13332-6-9-9-4) {\color{red} $\nu_{\Gamma_{\widetilde{w}}}$};
    
  \path[fill=black,line width=0.056pt] (76.8213,41.0198) node[above right]
    (text13332-6-9-6) {\color{black} $\nu_{\Gamma_{w}}$};
    
  \path[fill=blue,line width=0.056pt] (21.4000,26.2527) node[above right]
    (text13332-6-9-6-6-0) {\color{blue} $\nu_{\partial \Omega}$};    
    
  \path[fill=blue,line width=0.056pt] (91.7674,26.2849) node[above right]
    (text13332-6-9-6-6-0-9) {\color{blue} $\nu_{\partial \Omega}$};
    
  \path[fill=blue,line width=0.056pt] (25.3829,33.2527) node[above right]
    (text13332-6-9-6-6-0) {\color{blue} $S^+$};
    
  \path[fill=blue,line width=0.056pt] (88.7674,33.2849) node[above right]
    (text13332-6-9-6-6-0-9) {\color{blue} $S^-$};
    
  \path[fill=black,line width=0.056pt] (33.5481,21.2954) node[above right]
    (text13332-6-9-6-8) {$\Gamma_w$};
    
  \path[fill=red,line width=0.056pt] (50.8572,45.6409) node[above right]
    (text13332-6-9-6-8-8) {\color{red} $\Gamma_{\widetilde{w}}$};

\end{tikzpicture}
\caption{}
\label{Fig:LocalCalibration}
\end{figure}

Consider now the regions between the graphs of $w$ and $\widetilde{w}$,
distinguishing the parts above and below each function
\[
R^{+} = \{(x,\lambda)\in \Omega \times \R \colon \widetilde{w}(x) < \lambda < w(x)\},
\]
\[
R^{-} = \{(x,\lambda)\in \Omega \times \R \colon w(x) < \lambda < \widetilde{w}(x)\},
\]
as well as their lateral boundaries on $\partial \Omega \times \R$, that is,
$S^{+} = \overline{R^{+}} \cap (\partial \Omega \times \R)$ and $S^{-} = \overline{R^{-}} \cap (\partial \Omega \times \R)$;
see Figure~\ref{Fig:LocalCalibration}.
Applying the divergence theorem to the field $X$ separately in each of the regions $R^{+}$ and $R^{-}$, 
we see that
\[
\calibloc(w) = \calibloc(\widetilde{w}) + \int_{R^{+}} \!\! \div X  \d \mathcal{H}^{n+1} - \int_{R^{-}} \!\! \div X  \d \mathcal{H}^{n+1}
- \int_{S^{+}}\! \innerprod{X}{\nu_{\partial \Omega}} \d \mathcal{H}^{n} + \int_{S^{-}} \! \innerprod{X}{\nu_{\partial \Omega}} \d \mathcal{H}^{n},
\]
where we have extended the outer normal $\nu_{\partial \Omega}$ parallel to the surface $\partial \Omega \times \R$; see Figure~\ref{Fig:LocalCalibration}.
It is also immediate to check that 
\[ 
X \cdot \nu_{\partial \Omega} =  \neumannloc(u^{t})\big|_{t = t(x, \lambda))}
\]
on $S^{+}\cup S^{-}$, with $\neumannloc$ as in \eqref{loc:neumann}.
Thus, we obtain the passage from~\eqref{calib:loc} to~\eqref{calib:loc:2} as an application of the divergence theorem.
\end{remark}

Next we prove the key null-Lagrangian property~\ref{def:calib:1} for the calibration, which follows readily from the new identity~\eqref{calib:loc:2} for $\calibloc$:
\begin{proposition} \label{property:1:loc}

Under the same hypotheses as in Theorem~\ref{alt:expression:local},
assume that, for some $t_0 \in I$, the leaves of the field $\{u^{t}\}_{t \in I}$ satisfy
	\begin{equation} \label{sup:sub:loc}
		\begin{split}
			\oploc(u^{t}) \geq 0 \quad \text{ in } \Omega 
			\quad
			\text{ for } t \geq t_0,
			\\
			\oploc(u^{t}) \leq 0 \quad \text{ in } \Omega \quad \text{ for } t \leq t_0,
		\end{split}
	\end{equation}
	where $\oploc$ is the Euler-Lagrange operator introduced in \eqref{loc:operator}.

	Then, for all $w$ in $\admissibleloc$ such that $w \equiv u^{t_0}$ on $\partial\Omega$, 
	the functional $\calibloc$ defined in \eqref{calib:loc} satisfies
	\[
	\calibloc(u^{t_0})\leq \calibloc(w).
	\]
	
	Assume, moreover, that the leaves $\{u^{t}\}_{t\in I}$ satisfy the Euler-Lagrange equation in $\Omega$, that is,
	\begin{equation} \label{extremal:loc}
		\oploc(u^{t}) = 0 \quad \text{ in } \Omega \quad \text{ for all } t \in I.
	\end{equation}
	Then, for all $w$ as above, we have
	\[
	\calibloc(w) = \calibloc(u^{t_0}).
	\]
\end{proposition}

\begin{proof}
	Notice that,  by \eqref{calib:loc:2}, we have $\calibloc(u^{t_0}) = \energyloc(u^{t_0})$.
	Hence, assuming~\eqref{sup:sub:loc}, since the boundary integral in~\eqref{calib:loc:2} vanishes ($w \equiv u^{t_0}$ on $\partial \Omega$),
	it suffices to show that
\begin{equation}
\label{argo}
\int_{u^{t_0}(x)}^{w(x)} \oploc(u^{t})(x)\big|_{t = t(x, \lambda)} \d \lambda \geq 0.
\end{equation}
However, this is clear 
by~\eqref{sup:sub:loc} and the fact that $u^{t}$ are increasing with respect to $t$.
If we additionally have~\eqref{extremal:loc}, then the integral in \eqref{argo} is zero and the claim follows.
\end{proof}

The remaining calibration properties~\ref{def:calib:2} and~\ref{def:calib:3} can be directly obtained from the original definition~\eqref{calib:loc} of $\calibloc$.
First, we prove property~\ref{def:calib:2}:

\begin{proposition}
Assume the same hypotheses of Theorem~\ref{alt:expression:local}.
Then, for all $t \in I$, the functional $\calibloc$ defined in \eqref{calib:loc} satisfies
\[
\calibloc(u^{t}) = \energyloc(u^{t}).
\]

\end{proposition}
\begin{proof}
Given $t_0 \in I$, 
from the definition of the leaf-parameter function it follows that $t(x, u^{t_0}(x)) =t_0$. In particular, $\nabla u^{t_0}(x) = \nabla u^{t}(x)\big|_{t = t(x, u^{t_0}(x))}$,
	and substituting in the definition of $\calibloc$ in~\eqref{calib:loc} we see that
	\[
	\calibloc(u^{t_0}) = \int_{\Omega} \lagloc\big(x, u^{t_0}(x), \nabla u^{t_0}(x)\big) \d x = \energyloc(u^{t_0}).
	\]
	Since $t_0$ was arbitrary, the claim follows.
\end{proof}


Finally, we show property~\ref{def:calib:3}:
\begin{proposition}
Under the same hypotheses as in Theorem~\ref{alt:expression:local}, the energy $\energyloc$ can be decomposed in terms of $\calibloc$ and the excess function $\excess$ as
\begin{equation*}
	\energyloc(w) = \calibloc(w) + \int_{\Omega} \excess(x, u^t(x), \nabla u^t(x), \nabla w(x))\big|_{t = t(x, w(x))} \d x,\footnote{This identity is known as the \emph{Weierstrass representation formula}}
\end{equation*}
for all $w$ in $\admissibleloc$.

As a consequence, if each leaf of the field $\{u^{t}\}_{t\in I}$ satisfies the Weierstrass necessary condition~\eqref{weierstrass:nec}, then
\[
\calibloc(w) \leq \energyloc(w),
\]
for all $w$ in $\admissibleloc$.
\end{proposition}
\begin{proof}
For each $x\in \Omega$, take $\lambda = w(x) = u^{t}(x)|_{t = t(x, w(x))}$, $q=\nabla u^t(x)\big|_{t = t(x, w(x))}$, and $\widetilde{q} = \nabla w(x)$.
Substituting in \eqref{excess:function} and integrating in $\Omega$, 
comparing this expression with the definition of $\calibloc$ in \eqref{calib:loc},
the identity follows.
\end{proof}

\begin{remark}\label{Remark:Dirichlet-Neumann}
	If the energy functional includes reaction terms on a portion of the boundary $\Gamma_{\mathcal{N}} \subset \partial \Omega$, i.e.,
	\[\widetilde{\energyloc}(w) = \int_{\Omega} \lagloc \big(x, w(x), \nabla w(x)\big) \d x -\int_{\Gamma_{\mathcal{N}}} F(w(x)) \d \mathcal{H}^{n-1}(x),\]
	then we can still apply the methods from the local theory of calibrations to show that
	\begin{equation*}
		\begin{split}
			\widetilde{\calibloc}(w) = & \int_{\Omega} \innerprod{\partial_q \lagloc(x, u^{t}(x), \nabla u^{t}(x))}{\big(\nabla w(x) -\nabla u^{t}(x)\big)}\big|_{t = t(x, w(x))} \d x\\
			& \hspace{1cm} + \int_{\Omega} \lagloc(x, u^{t}(x), \nabla u^{t}(x))\big|_{t = t(x, w(x))} -\int_{\Gamma_{\mathcal{N}}} F(w(x)) \d \mathcal{H}^{n-1}(x) \d x
		\end{split}
	\end{equation*}
	is a calibration. Hence, one can establish the minimality of the leaves among competitors with the same boundary data only on $\partial \Omega \setminus\Gamma_{\mathcal{N}}.$
	Note that in this scenario, extremals satisfy the equation
	\[ \left \{ \begin{array}{r r l l}
		\oploc(u) &=& 0 & \text{ in } \Omega,\\
		\neumannloc(u) &=& F'(u) & \text{ in } \Gamma_{\mathcal{N}}.
	\end{array}\right .\]
	In particular, this allows to treat the extension problem for the fractional Laplacian as explained in Appendix~\ref{section:extension}.
	In this setting, one considers the Dirichlet energy in a domain of the extended space $\R^{n+1}_{+}$ with an additional potential energy on the part of its boundary lying on $\R^{n} = \partial \R^{n+1}_{+}.$
\end{remark}


\section{The calibration for the fractional Laplacian}
\label{section:fractional:laplacian}
In this section we construct a calibration for the functional
\begin{equation*}
\energyfrac(w) = 
\dfrac{c_{n,s}}{4}\iint_{\domain} \dfrac{\abs{w(x)-w(y)}^2}{\abs{x-y}^{n+2s}} \d x \d y - \int_{\Omega} F(w(x)) \d x,
\end{equation*}
where 
$s \in (0,1)$, $c_{n,s}$ is a positive normalizing constant, and $F \in C^1(\R)$. It involves the Gagliardo-Sobolev seminorm and a nonlinear potential term.

If $u$ is a 
critical point
of $\energyfrac$ with respect to functions with the same exterior data as~$u$, then $u$ satisfies the nonlocal semilinear equation 
\begin{equation}\label{euler:lagrange:frac}
	\fraclaplacian u - F'(u) = 0 \quad \text{ in } \Omega.
\end{equation}

In this setting, recall the standard subspace of locally integrable functions given by
\[
\lsn
:= \Big\{
u \in L^{1}_{\rm loc}(\R^n)
\colon 
\|u\|_{
	\lsn
} = \int_{\R^n} \frac{|u(y)|}{1+|y|^{n+2s}} \d y < +\infty\Big\}.\]

If $u \in \lsn$ is $C^{2}$ in a neighborhood of $x \in \R^n$, then 
the fractional Laplacian $\fraclaplace u(x)$
is well defined.
More generally, we only need $u$ to be $C^{2s+\alpha}$ in a neighborhood of $x$ for some small $\alpha > 0$ such that $2s+\alpha$ is not an integer. Here $C^\beta$ denotes the space $C^{k, \gamma}$ of functions with H\"{o}lder continuous $k$-th order derivatives, where $k = \lfloor\beta\rfloor$, $\gamma = \beta - k$.
In particular, for $\fraclaplace u$ to be well defined in a domain $\Omega$,
we just need that the function $u \in \lsn$ is smooth in a neighborhood of $\Omega$.
The function could be extremely wild outside the domain, as long as it satisfies the growth assumption defining $L_s^1(\R^n)$.

As explained in the Introduction, to build the calibration we will assume the existence of a field in~$\R^n$. 
In particular, we are given a family of functions $u^{t} \colon \R^n \to \R$, with $t \in I$ for some interval $I \subset \R$, satisfying certain regularity assumptions.
For clarity reasons, in the statement of Theorem~\ref{thm:frac:calibration} we have assumed that the function $(x, t)\mapsto u^t(x)$ belongs to $C^2(\R^n\times I)\cap L^{\infty}(\R^n\times I)$
and $(x, t)\mapsto \partial_t u^t(x)$ is in $L^{\infty}(\R^n \times I)$.
We have also assumed that $t \mapsto u^{t}(x)$ is increasing in $I$ for all $x \in \R^{n}$.
Thus, the graphs of $u^t$ produce a foliation of a certain region in $\R^{n} \times \R$.
Nevertheless, these conditions can be weakened, as presented next, to yield a more satisfactory theory for the fractional Laplacian.

The following is a weaker definition of field than the one in the statement of Theorem~\ref{thm:frac:calibration}, but which suffices to establish the result.
On the one hand, we allow the functions $u^t$ in the field to be ``wild'' outside a neighborhood of $\Omega$, as long as their fractional Laplacian is under control.
On the other hand, the leaves $u^t$  can touch each other (but not cross) outside $\Omega$.
That is, we need $u^t$ to be increasing in $t$ in $\overline{\Omega}$, but only nondecreasing outside.
In particular, the graphs of $u^t$ will only produce a foliation in a certain region of $\overline{\Omega}\times \R$.

\begin{definition}
	\label{def:field:frac}
	Given an interval $I \subset \R$, a  bounded domain $\Omega \subset \R^n$, and $s \in (0,1)$,
	we say that a family $\{u^{t}\}_{t \in I}$ of functions $u^{t}\colon \R^n \to \R$ is a \emph{field} (to be precise, we should say a \emph{field associated to the $s$-fractional Laplacian}, not to be inconsistent with Definition~\ref{def:foliation})
	when the following conditions hold:
	\begin{enumerate}[label= (\roman*)]
		\item \label{def:field:1} The function $(x, t) \mapsto u^{t}(x)$ is continuous in $\overline{\Omega} \times I$.
		\item \label{def:field:2} 
		The function $t \mapsto u^{t}(x)$ is
		\begin{itemize}
			\item increasing in $I$ for all $x \in \overline{\Omega}$
			\item nondecreasing in $I$ for a.e. $x \in \Omega^c$
			\item $C^1$ in $I$ for a.e. $x \in \R^n$
		\end{itemize}
		\item \label{def:field:3}
		For each compact interval $J \subset I$, we have
		\[
		\sup_{t \in J} \,\Big\{\|\partial_t u^{t}\|_{L^{\infty}(\R^n)} + \|u^{t}\|_{\lsn} + \|u^{t}\|_{C^{2s+\alpha}(N)}\Big\}  < \infty,
		\]
		for some bounded domain $N \subset \R^n$, with $\overline{\Omega} \subset N$, 
		and some $\alpha > 0$.
	\end{enumerate}
\end{definition}

Essentially, one needs a reasonable regularity of the joint function, as well as some further regularity separately in each of the variables, locally uniformly in the parameter~$t$.

By properties \ref{def:field:1} and \ref{def:field:2}, the leaf-parameter function $t = t(x, \lambda)$ 
from~\eqref{Eq:DefLeafParameter} is well-defined and continuous in the region
\begin{equation}\label{def:regiongag}
	\region := \{ (x, \lambda) \in \overline{\Omega} \times \R \colon \lambda = u^{t}(x) \text{ for some } t \in I\}.
\end{equation}

The potential $F$ will play no role in the construction of the calibration for $\energyfrac$ and,
hence, we can focus on the first term.
Let
\[
\gagliardo(u) := 
\frac{c_{n,s}}{4} \iint_{\domain} 
\dfrac{\abs{u(x)-u(y)}^2}{\abs{x-y}^{n+2s}}
\d x \d y
\]
and consider the energy space $\hgag:= \{u \in L^{1}_{\rm loc}(\R^{n}) \colon \gagliardo(u) < \infty\}$.

Let $\{u^{t}\}_{t \in I}$ be a field in the sense of Definition \ref{def:field:frac}, and let $t_0 \in I$.
We consider the functional
\begin{equation}
	\label{frac:calib}
	\calibgags(w) := \int_{\Omega} \int^{w(x)}_{u^{t_0}(x)} \fraclaplace u^{t}(x) \big|_{t = t(x, \lambda)} \d \lambda \d x 
	+ 
	\frac{c_{n,s}}{4} \iint_{\domain} \dfrac{\abs{u^{t_0}(x)-u^{t_0}(y)}^2}{\abs{x-y}^{n+2s}} \d x \d y
\end{equation}
acting on continuous functions $w \in C^0(\overline{\Omega})$ with the same exterior datum as $u^{t_0}$, and such that their graph is contained in $\region \subset \overline{\Omega}\times \R$
when restricted to $\overline{\Omega}$.
Here $\region$ has been introduced in \eqref{def:regiongag}.
We denote this set of admissible functions by $\hbdy$, that is,
\begin{equation}
\label{def:admissible:frac}
\hbdy := \big\{w \in C^0(\overline{\Omega}) \colon w = u^{t_0} \text{ on } \partial\Omega, \ \ w = u^{t_0} \text{ a.e. in } \Omega^c, \, {\rm graph }\, (w_{\arrowvert \overline{\Omega}})  \subset \region \big\}.
\end{equation}

\begin{remark}
	The functional $\calibgags$ is well defined in the set $\hbdy$.
	Let us check this. For~$x \in \Omega$ and~$\lambda$ between $u^{t_0}(x)$ and $w(x)$, we have that 
	$$t(x,\lambda) \in [t_{\rm min}, t_{\rm max}] \subset I,$$
	where
	\begin{equation} \label{def:tmin:tmax}
		t_{\rm min} = \min_{x \in \overline{\Omega}} t(x,w(x)) \ \ \text{ and } \ \ t_{\rm max} = \max_{x \in \overline{\Omega}} t(x,w(x)).\footnote{Note here that $t(x,u^{t_0}(x))\equiv t_0$ and that $t_{min} \leq t_0 \leq t_{max}$ since $w\equiv u^{t_0}$ on $\partial \Omega$.}
	\end{equation}
	Then, on the one hand, since the fractional Laplacians $\fraclaplacian u^{t}(x)$ are uniformly bounded in $x \in \Omega$ and $t \in [t_{\rm min}, t_{\rm max}]$ by \ref{def:field:3}, the iterated integral in the first term in~\eqref{frac:calib} is finite.
	On the other hand, taking into account the identity
	\[
	\frac{c_{n,s}}{2}\iint_{\domain} \frac{|u^{t}(x) - u^{t}(y)|^2}{|x-y|^{n+2s}} \d x \d y = \int_{\Omega} u^{t}(x) \fraclaplacian u^{t}(x) \d x,
	\]
	the second integral in~\eqref{frac:calib} is finite thanks to the uniform boundedness in $x \in \Omega$ and $t \in [t_{\rm min}, t_{\rm max}]$ of each $u^{t}(x)$ and of the fractional Laplacians $\fraclaplacian u^{t}(x)$.
\end{remark}

\begin{remark}
\label{remark:pv}
	The functional $\calibgags$ coincides, in the set $\hbdy$, with the functional $\calibfrac$ appearing in Theorem~\ref{thm:frac:calibration} when $F = 0$. Indeed,
	we can write
	\begin{align*}
		&\int_{\Omega} \int^{w(x)}_{u^{t_0}(x)} \fraclaplace u^{t}(x) \big|_{t = t(x, \lambda)} \d \lambda \d x \\
		&\hspace{2cm}= 
		\int_{\Omega} \d x \int^{w(x)}_{u^{t_0}(x)} \d \lambda \lim_{\varepsilon \downarrow 0} \int_{\R^n \setminus \{|x-y| > \varepsilon \}} \d y \, c_{n,s} \frac{u^{t}(x) - u^{t}(y)}{|x-y|^{n+2s}} \bigg|_{t = t(x, \lambda)} \\
		&\hspace{2cm}= 
		\lim_{\varepsilon \downarrow 0}
		\iint_{(\Omega \times \R^n) \setminus \{|x-y|> \varepsilon \}} \d x \d y\int^{w(x)}_{u^{t_0}(x)} \d \lambda 
		\, c_{n,s} \frac{u^{t}(x) - u^{t}(y)}{|x-y|^{n+2s}} \bigg|_{t = t(x, \lambda)} \\
		&\hspace{2cm}=  c_{n,s} \
		\lim_{\varepsilon \downarrow 0}
		\iint_{Q(\Omega) \setminus \{|x-y|> \varepsilon \}} \d x \d y\int^{w(x)}_{u^{t_0}(x)} \d \lambda 
		\,  \frac{u^{t}(x) - u^{t}(y)}{|x-y|^{n+2s}} \bigg|_{t = t(x, \lambda)}
	\end{align*}
	by the regularity of the field and where the last equality follows from the fact that $u^{t_0}(x) = w(x)$ for almost every $x\in \Omega^c$.
	This last expression gives meaning to the principal value in the definition of $\calibfrac$ in 
	Theorem~\ref{thm:frac:calibration} in the Introduction.
\end{remark}

Notice that the expression of $\calibgags$ in \eqref{frac:calib} only involves the Euler-Lagrange equation of the field $\{u^{t}\}_{t \in I}$ and the energy of the leaf $u^{t_0}$.
This will give both~\ref{def:calib:2} and the null-Lagrangian property~\ref{def:calib:1}. To prove \ref{def:calib:3}, we next show, in Lemma~\ref{lemma:equality}, that the functional \eqref{frac:calib} can be recast in a useful alternative form. We include Figure~\ref{Fig:Dibujo_Calibracion} for the convenience of the reader, to better identify the terms involved in the new expression~(of Lemma~\ref{lemma:equality}).

To simplify the statements and proofs below, for $\varepsilon > 0$ we use the truncated kernel
$K_{\varepsilon}(z) = c_{n,s} |z|^{-n-2s} \, \mathds{1}_{B_{\varepsilon}^c}(z)$, 
and for $u \in \lsn$ we let
$$\fraclaplace_\varepsilon u(x) = \int_{\R^n} (u(x) - u(y)) K_{\varepsilon}(x-y) \d y.$$
In particular, when $\varepsilon$ goes to zero we recover the fractional Laplacian
$\fraclaplacian u^{t}(x) = \lim_{\varepsilon \downarrow 0} \fraclaplacian_{\varepsilon}u^{t}(x)$.
We also write
\begin{equation}
\label{def:calibeps}
\begin{split}
\calibgageps(w) := &\int_{\Omega} \int^{w(x)}_{u^{t_0}(x)} \!\!\! \fraclaplace_{\varepsilon} u^{t}(x) \big|_{t = t(x, \lambda)} \d \lambda \d x 
\\	
&\hspace{3cm} + 
\frac{1}{4} \iint_{\domain} \!\!\!  |u^{t_0}(x)-u^{t_0}(y)|^2 K_{\varepsilon}(x-y) \d x \d y.
\end{split}
\end{equation}

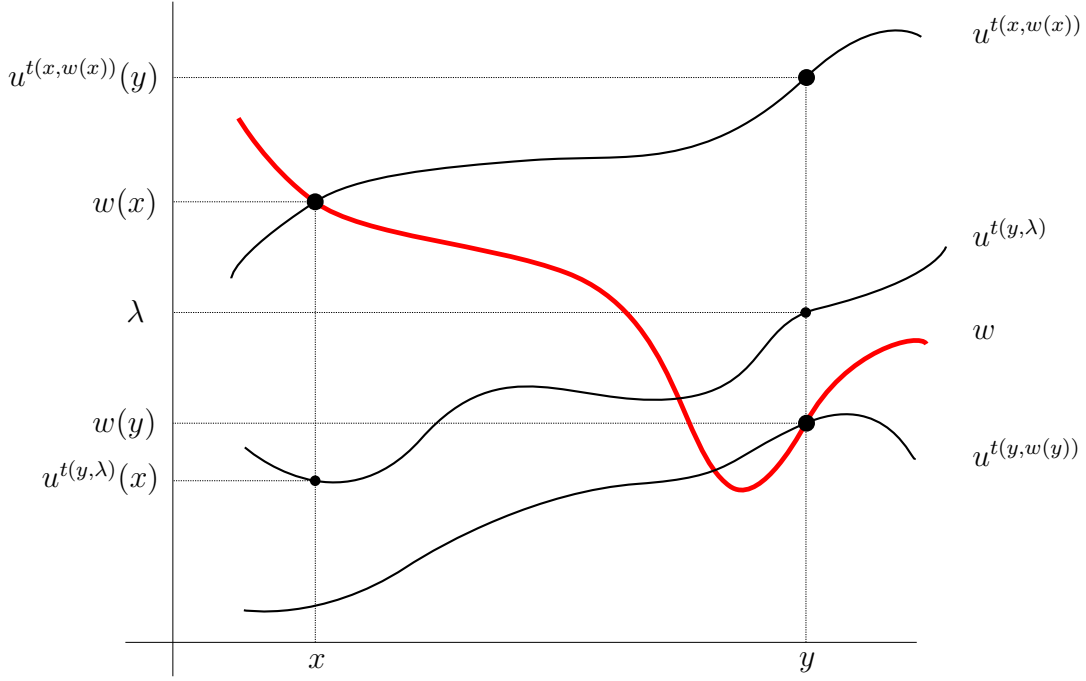
\begin{figure}
	\centering
	\begin{tikzpicture}[y=0.80pt, x=0.80pt, yscale=-4.000000, xscale=4.000000, inner sep=0pt, outer sep=0pt]
	\path[draw=black,line join=miter,line cap=butt,line width=0.056pt]
	(38.4868,38.1269) -- (38.4868,117.9294);
	
	\path[draw=black,line join=miter,line cap=butt,line width=0.056pt]
	(32.8742,113.9204) -- (126.4185,113.9204);
	
	\path[draw=black,dash pattern=on 0.45pt off 0.45pt,line join=miter,line
	cap=butt,miter limit=4.00,line width=0.056pt] (55.3248,113.9249) --
	(55.3248,61.7854);
	
	\path[draw=black,dash pattern=on 0.45pt off 0.45pt,line join=miter,line
	cap=butt,miter limit=4.00,line width=0.056pt] (113.3223,113.9162) --
	(113.3223,47.1022);
	
	\path[draw=black,dash pattern=on 0.45pt off 0.45pt,line join=miter,line
	cap=butt,miter limit=4.00,line width=0.056pt] (113.3471,47.1030) --
	(38.4891,47.1030);
	
	\path[draw=black,dash pattern=on 0.45pt off 0.45pt,line join=miter,line
	cap=butt,miter limit=4.00,line width=0.056pt] (55.3115,61.8028) --
	(38.4772,61.8028);
	
	\path[draw=black,dash pattern=on 0.45pt off 0.45pt,line join=miter,line
	cap=butt,miter limit=4.00,line width=0.056pt] (113.2939,74.8990) --
	(38.5052,74.8990);
	
	\path[draw=black,dash pattern=on 0.45pt off 0.45pt,line join=miter,line
	cap=butt,miter limit=4.00,line width=0.056pt] (113.3874,87.9952) --
	(38.4935,87.9952);
	
	
	\path[draw=red,line join=miter,line cap=butt,line width=1.8pt]
	(46.2376,51.9138) .. controls (46.2376,51.9138) and (49.6070,57.7170) ..
	(55.2914,61.8362) .. controls (60.9758,65.9554) and (75.9410,66.7603) ..
	(85.0581,70.2261) .. controls (98.1694,75.2104) and (98.4305,91.2520) ..
	(104.3005,95.4329) .. controls (107.5614,97.7554) and (111.7248,90.9278) ..
	(113.3223,87.9618) .. controls (118.0410,79.2010) and (126.7906,77.1965) ..
	(127.5611,78.6274);
	
	\path[draw=black,line join=miter,line cap=butt,line width=0.8pt]
	(126.9530,42.2921) .. controls (126.9530,42.2921) and (122.1601,38.7042) ..
	(113.3013,47.1016) .. controls (101.2381,58.5367) and (92.3992,56.0298) ..
	(80.8902,56.8783) .. controls (69.3813,57.7269) and (59.9135,58.7926) ..
	(55.2830,61.8258) .. controls (45.2445,68.4016) and (45.4024,70.8566) ..
	(45.4024,70.8566);
	
	\path[draw=black,line join=miter,line cap=butt,line width=0.8pt]
	(129.8863,67.0978) .. controls (129.8863,67.0978) and (128.9474,70.9513) ..
	(114.6738,74.4582) .. controls (107.0279,76.3367) and (108.8296,84.3099) ..
	(97.3206,85.1585) .. controls (85.8117,86.0070) and (77.5567,78.8551) ..
	(67.8380,89.8387) .. controls (58.1194,100.8223) and (46.9994,90.8514) ..
	(46.9994,90.8514);
	
	\path[draw=black,line join=miter,line cap=butt,line width=0.8pt]
	(126.1889,92.1123) .. controls (126.1259,93.2245) and (123.1029,83.5212) ..
	(112.7261,88.2284) .. controls (103.7214,92.3131) and (104.9587,94.3217) ..
	(93.4497,95.1702) .. controls (81.9407,96.0188) and (70.6776,101.9982) ..
	(66.0471,105.0314) .. controls (56.0086,111.6072) and (46.8795,110.1226) ..
	(46.8795,110.1226);

	\path[fill=black,line width=0.056pt] (33,76) node[above right](text4563) {$\lambda$};
	\path[fill=black,line width=0.056pt] (29,90) node[above right](text4563) {$w(y)$};
	
	\path[fill=black,line width=0.056pt] (23,96.5) node[above right](text4563) {$u^{t(y,\lambda)}(x)$};
	
	\path[fill=black,line width=0.056pt] (29,64) node[above right](text4563) {$w(x)$};
	\path[fill=black,line width=0.056pt] (19,49) node[above right](text4563) {$u^{t(x,w(x))}(y)$};
	
	\path[fill=black,line width=0.056pt] (54.5,117) node[above right](text4563) {$x$};
	\path[fill=black,line width=0.056pt] (112.5,117.5) node[above right](text4563) {$y$};
	
	\path[fill=black,line width=0.056pt] (133,43) node[above right](text4563) {$u^{t(x,w(x))}$};
	\path[fill=black,line width=0.056pt] (133,67) node[above right](text4563) {$u^{t(y,\lambda)}$};
	\path[fill=black,line width=0.056pt] (133,78) node[above right](text4563) {$w$};
	\path[fill=black,line width=0.056pt] (133,93) node[above right](text4563) {$u^{t(y,w(y))}$};
	
	\path[fill=black,dash pattern=on 0.76pt off 0.76pt,line cap=round,miter
	limit=4.00,line width=0.095pt] (55.3115,61.8028) circle (0.8pt);
	
	\path[fill=black,dash pattern=on 0.76pt off 0.76pt,line cap=round,miter
	limit=4.00,line width=0.095pt] (113.2939,74.8990) circle (0.5pt);
	
	\path[fill=black,dash pattern=on 0.76pt off 0.76pt,line cap=round,miter
	limit=4.00,line width=0.095pt] (113.3874,87.9952) circle (0.8pt);
	
	\path[fill=black,dash pattern=on 0.76pt off 0.76pt,line cap=round,miter
	limit=4.00,line width=0.095pt] (113.3874,47.1022) circle (0.8pt);
	
	\path[fill=black,dash pattern=on 0.76pt off 0.76pt,line cap=round,miter
	limit=4.00,line width=0.095pt] (55.3115,94.8) circle (0.5pt);
	
	\path[draw=black,dash pattern=on 0.45pt off 0.45pt,line join=miter,line
	cap=butt,miter limit=4.00,line width=0.056pt] (55.3115,94.8) --
	(38.5052,94.8);

\end{tikzpicture}
	\caption{The function $w$ (in red) and the leaves $u^t$.}
	\label{Fig:Dibujo_Calibracion}
\end{figure}

\begin{lemma}
	\label{lemma:equality}
	Given an interval $I\subset \R$, a bounded domain
	$\Omega \subset \R^n$,
	and $s \in (0,1)$,
	let $\{u^{t}\}_{t \in I}$ be a field in the sense of Definition \ref{def:field:frac}.
	Consider
	the set of admissible functions~$\hbdy$ defined in \eqref{def:admissible:frac}.
	
	Then,
	for each $\varepsilon > 0$ and $w \in \hbdy$, with $t_0 \in I$,
	the functional $\calibgageps$ defined in \eqref{def:calibeps}
	satisfies
	\[
	\begin{split}
		\calibgageps(w) = &- \frac{c_{n,s}}{2}\iint_{\domain \setminus\{ |x-y| < \varepsilon\}} \d x \d y \int_{t(x,w(x))}^{t(y,w(y))} \dfrac{u^{t}(x) - u^{t}(y)}{|x-y|^{n+2s}}\partial_{t}u^{t}(y) \d t\\
		&+ \frac{c_{n,s}}{4}\iint_{\domain \setminus \{|x-y| < \varepsilon \}} \frac{|w(x) - u^{t(x,w(x))}(y)|^2}{|x-y|^{n+2s}}  \d x \d y,
	\end{split}
	\]
	where we have extended the leaf-parameter function $x \mapsto t(x, w(x))$ continuously outside~$\Omega$ 
	by letting $t(x, w(x)) = t_0$ 
	for 
	$x \in \Omega^c$.
\end{lemma}

\begin{proof}
	Throughout the proof, $C_{\varepsilon}$ denotes a generic positive constant
	depending only on $n$, $s$, $|\Omega|$, and $\varepsilon$.
	It is easy to check that
	\begin{equation}
		\label{use:growth}
		\left(1 + |y|^{n+2s}\right)\,K_{\varepsilon}(x-y) \leq 
		C_{\varepsilon}
		\quad \text{ for all } x \in \Omega \text{ and } y \in \R^n.
	\end{equation}
	By \eqref{use:growth}
	we have
	\[
	K_{\varepsilon}(x-y)\, |u^{t}(x) - u^{t}(y)|  \leq 
	C_{\varepsilon}\left(\|u^{t}\|_{C^0(\overline{\Omega})} + |u^{t}(y)|\right)\left(1 + |y|^{n+2s}\right)^{-1}
	\]
	for all $x\in \Omega$, a.e. $y \in \R^n$ and all $t \in I$. 
	Hence, integrating in $y \in \R^n$,
	\begin{equation}\label{unif:aux:1}
		\int_{\R^n} \! \! K_{\varepsilon}(x-y)\,|u^{t}(x)-u^{t}(y)|  \d y
		\leq C_{\varepsilon}\left(\|u^{t}\|_{C^0(\overline{\Omega})} + \|u^{t}\|_{\lsn}\right) \ \text{ for all } x \in \Omega, \, t \in I.
	\end{equation}
	Consider now $t_{\rm min}$ and $t_{\rm max }$ given in~\eqref{def:tmin:tmax}. By properties \ref{def:field:1} and \ref{def:field:3} in Definition~\ref{def:field:frac}, 
	for $x \in \Omega$
	we have
	\begin{align*}
		&\left|\int^{w(x)}_{u^{t_0}(x)} \d \lambda \int_{\R^n} \d y \, K_{\varepsilon}(x-y)\,|u^{t}(x)-u^{t}(y)| \big|_{t = t(x, \lambda)}
		\right| \\
		&\hspace{4cm} \leq C_{\varepsilon}
		\sup_{t \in [t_{\rm min}, t_{\rm max}]}\left(\|u^{t}\|_{C^0(\overline{\Omega})} + \|u^{t}\|_{\lsn}\right) \|w - u^{t_0}\|_{C^0(\overline{\Omega})} < \infty,
	\end{align*}
	and we can apply Fubini's theorem to get
	\begin{equation}\label{eq:ex:1}
		\begin{split}
			&\int_{\Omega} \d x \int^{w(x)}_{u^{t_0}(x)} \d \lambda \fraclaplace_{\varepsilon} u^{t}(x) \big|_{t = t(x, \lambda)} \\ 
			&\hspace{3cm}= \int_{\Omega} \d x \int_{\R^n} \d y K_{\varepsilon}(x-y) \int^{w(x)}_{u^{t_0}(x)}  (u^{t}(x)-u^{t}(y)) \big|_{t = t(x, \lambda)} \d \lambda.
		\end{split}
	\end{equation}
	Applying the change of variables $\lambda = u^{t}(x)$
	for a.e. $x \in \Omega$
	in \eqref{eq:ex:1}, 
	the integral becomes
	\begin{equation}\label{eq:ex:2}
		\begin{split}
			&\int_{\Omega} \d x \int_{\R^n}\d y \ K_{\varepsilon}(x-y) \int^{w(x)}_{u^{t_0}(x)}  (u^{t}(x)-u^{t}(y)) \big|_{t = t(x, \lambda)}\d \lambda \\
			&\hspace{3cm}= \int_{\Omega} \d x \int_{\R^n}\d y \ K_{\varepsilon}(x-y) \int^{t(x,w(x))}_{t_0}  \!(u^{t}(x)-u^{t}(y)) \, \partial_t u^t(x) \d t.
		\end{split}
	\end{equation}
	Thanks to the extension of the leaf-parameter function by $t(x, w(x)) = t_0$ for $x \in \Omega^c$, 	using that $\domain = (\Omega \times \R^n) \cup (\Omega^c \times \Omega)$,
	we can rewrite the right-hand side of \eqref{eq:ex:2} as
	\begin{equation}\label{eq:ex:3}
		\begin{split}
			& \int_{\Omega} \d x \int_{\R^n}\d y \ K_{\varepsilon}(x-y) \int^{t(x,w(x))}_{t_0}  \! (u^{t}(x)-u^{t}(y)) \, \partial_t u^t(x) \d t\\
			&\hspace{3cm}= \iint_{\domain} \d x \d y\ K_{\varepsilon}(x-y) \int^{t(x,w(x))}_{t_0}  \! (u^{t}(x)-u^{t}(y)) \ \partial_t u^t(x) \d t.
		\end{split}
	\end{equation}
	
	The idea now is to use the symmetry of the domain $\domain$ to symmetrize the right-hand side of the previous identity. This will allow us to integrate an exact differential of $t$, and this will lead to the identity claimed in the lemma.
	
	Symmetrizing \eqref{eq:ex:3}, we have
	\begin{equation}\label{eq:ex:4}
		\begin{split}
			&\iint_{\domain} \d x \d y K_{\varepsilon}(x-y) \int^{t(x,w(x))}_{t_0}  (u^{t}(x)-u^{t}(y)) \partial_tu^t(x) \d t\\
			&\hspace{2cm} = \frac{1}{2}\iint_{\domain} \d x \d y K_{\varepsilon}(x-y) \int^{t(x,w(x))}_{t_0}  (u^{t}(x)-u^{t}(y)) \partial_tu^t(x) \d t\\
			&\hspace{2cm} \quad - \frac{1}{2}\iint_{\domain} \d x \d y K_{\varepsilon}(x-y) \int^{t(y,w(y))}_{t_0}  (u^{t}(x)-u^{t}(y)) \partial_t u^t(y) \d t.
		\end{split}
	\end{equation}
	Splitting the integral $\int_{t_0}^{t(y,w(y))} \d t$ into the sum $\int_{t_0}^{t(x,w(x))} \d t + \int_{t(x, w(x))}^{t(y,w(y))} \d t$ and rearranging terms, the right-hand side of \eqref{eq:ex:4} becomes
	\begin{equation}\label{eq:ex:5}
		\begin{split}
			&\frac{1}{2}\iint_{\domain} \d x \d y K_{\varepsilon}(x-y) \int^{t(x,w(x))}_{t_0}  (u^{t}(x)-u^{t}(y)) \partial_tu^t(x) \d t\\
			&\hspace{1cm} \ \ \quad - \frac{1}{2}\iint_{\domain} \d x \d y K_{\varepsilon}(x-y) \int^{t(y,w(y))}_{t_0}  (u^{t}(x)-u^{t}(y)) \partial_t{u}^t(y) \d t\\
			&\hspace{1cm}= - \frac{1}{2}\iint_{\domain} \d x \d y K_{\varepsilon}(x-y) \int^{t(y,w(y))}_{t(x, w(x))}  (u^{t}(x)-u^{t}(y)) \partial_t{u}^t(y) \d t\\
			&\hspace{1cm}\ \ \quad +\frac{1}{2}\iint_{\domain} \d x \d y K_{\varepsilon}(x-y) \int^{t(x,w(x))}_{t_0}  (u^{t}(x)-u^{t}(y)) (\partial_t{u}^t(x) - \partial_t{u}^{t}(y) )\d t.
		\end{split}
	\end{equation}

	Let us show that the integrals in the right-hand side of \eqref{eq:ex:5} are well defined.
	For the first integral, 
	taking absolute values
	and using Fubini's theorem, we have
	\begin{equation}
		\label{less0}
		\begin{split}
			&\frac{1}{2}\iint_{\domain} \d x \d y K_{\varepsilon}(x-y) \left|\int_{t(x,w(x))}^{t(y,w(y))} |u^{t}(x) - u^{t}(y)| |\partial_{t}u^{t}(y)| \d t\right| \\
			& \quad \leq \frac{1}{2}\iint_{\domain} \d x \d y K_{\varepsilon}(x-y) \int_{t_{\rm min}}^{t_{\rm max}} |u^{t}(x) - u^t(y)| 
			\|\partial_t u^t\|_{L^{\infty}(\R^n)} \d t \\
			&\hspace{1cm} = 
			\frac{1}{2}\int_{t_{\rm min}}^{t_{\rm max}} \d t \, \|\partial_t u^t\|_{L^{\infty}(\R^n)}  \iint_{\domain}   K_{\varepsilon}(x-y)  |u^{t}(x) - u^{t}(y)|  \d x \d y\\
			&\hspace{1.5cm} \leq
			\int_{t_{\rm min}}^{t_{\rm max}} \d t \, \|\partial_t u^t\|_{L^{\infty}(\R^n)}  \int_{\Omega}\d x\int_{\R^n} \d y   K_{\varepsilon}(x-y)  |u^{t}(x) - u^{t}(y)|,
		\end{split}
	\end{equation}
	where in the last line we have used 
	that $\domain = (\Omega \times \R^n) \cup (\R^n \times \Omega)$ (not a disjoint union)
	and the symmetry of $K_{\varepsilon}$.
	Applying the bound \eqref{unif:aux:1} in \eqref{less0} and by property \ref{def:field:3} in Definition~\ref{def:field:frac}, we deduce the finiteness of \eqref{less0}.
	It follows that the first integral is well defined.

	The second integral in the right-hand side of \eqref{eq:ex:5} can be integrated explicitly as
	\begin{equation}\label{eq:ex:6}
		\begin{split}
			&\frac{1}{2}\iint_{\domain} \d x \d y K_{\varepsilon}(x-y) \int^{t(x,w(x))}_{t_0}  (u^{t}(x)-u^{t}(y)) (\partial_t{u}^t(x) - \partial_t{u}^{t}(y) )\d t\\
			&\hspace{3cm}= \frac{1}{4}\iint_{\domain} \d x \d y K_{\varepsilon}(x-y) \int^{t(x,w(x))}_{t_0}  \frac{\d}{\d t}|u^{t}(x)-u^{t}(y)|^2 \d t\\
			&\hspace{3cm}= \frac{1}{4}\iint_{\domain} \d x \d y K_{\varepsilon}(x-y) |w(x)-u^{t(x,w(x))}(y)|^2\\
			&\hspace{3cm} \quad \quad -\frac{1}{4}\iint_{\domain} \d x \d y K_{\varepsilon}(x-y) |u^{t_0}(x)-u^{t_0}(y)|^2.
		\end{split}
	\end{equation}
	Concatenating the equalities \eqref{eq:ex:1}, \eqref{eq:ex:2}, \eqref{eq:ex:3}, \eqref{eq:ex:4}, 
	and \eqref{eq:ex:5}, and using \eqref{eq:ex:6}, we conclude
	\[
	\begin{split}
		&\int_{\Omega} \d x \int^{w(x)}_{u^{t_0}(x)} \d \lambda \fraclaplace_{\varepsilon} u^{t}(x) \big|_{t = t(x, \lambda)} \\ 
		&\hspace{3cm}= - \frac{1}{2}\iint_{\domain} \d x \d y K_{\varepsilon}(x-y) \int^{t(y,w(y))}_{t(x, w(x))}  (u^{t}(x)-u^{t}(y)) \partial_t{u}^t(y) \d t\\
		&\hspace{3cm} \quad+ \frac{1}{4}\iint_{\domain} \d x \d y K_{\varepsilon}(x-y) |w(x)-u^{t(x,w(x))}(y)|^2 \\
		&\hspace{3cm}\quad -\frac{1}{4}\iint_{\domain} \d x \d y K_{\varepsilon}(x-y) |u^{t_0}(x)-u^{t_0}(y)|^2,
	\end{split}
	\]
	which is the claim of the lemma.
\end{proof}

Having Lemma~\ref{lemma:equality} at hand, we can now prove the calibration property
\ref{def:calib:3}.

\begin{lemma}
	\label{lemma:inequality}
	Given an interval $I\subset \R$, a bounded domain
	$\Omega \subset \R^n$,
	and $s \in (0,1)$,
	let $\{u^{t}\}_{t \in I}$ be a field in the sense of Definition \ref{def:field:frac}.
	Consider the set of admissible functions $\hbdy$ defined in \eqref{def:admissible:frac}.
	
	Then, for all $w \in \hbdy$, with $t_0 \in I$, we have
	\[
	\calibgags(w) \leq \gagliardo(w).
	\]
\end{lemma}\begin{proof}
	Let $\varepsilon > 0$.
	By Lemma \ref{lemma:equality}, we have
	\begin{equation}\label{eq:ex:7}
		\begin{split}
			\calibgageps(w) = &- \frac{1}{2}\iint_{\domain} \d x \d y K_{\varepsilon}(x-y) \int_{t(x,w(x))}^{t(y,w(y))} (u^{t}(x) - u^{t}(y))\partial_{t}u^{t}(y) \d t\\
			&+ \frac{1}{4}\iint_{\domain} |w(x) - u^{t(x,w(x))}(y)|^2K_{\varepsilon}(x-y) \d x \d y, \\
		\end{split}
	\end{equation}
	where $t(x, w(x)) = t_0$ for $x \in \Omega^c$.
	We claim that for a.e. $(x, y) \in \domain$ we have
	\begin{equation}\label{eq:ex:8}
		\begin{split}
			-\frac{1}{2}\int^{t(y,w(y))}_{t(x, w(x))}  (u^{t}(x)-u^{t}(y)) \partial_t{u}^t(y) \d t
			\leq -\frac{1}{2}\int^{t(y,w(y))}_{t(x, w(x))}  (w(x)-u^{t}(y)) \partial_t{u}^t(y) \d t.
		\end{split}
	\end{equation}
	Indeed, first of all note that $\partial_t{u}^{t}(y) \geq 0$ for a.e. $y\in \R^n$ by property~\ref{def:field:2} in Definition~\ref{def:field:frac}.
	Moreover, the quantities $u^{t}(x)$, $u^{t}(y)$, $w(x)$, and $\partial_t u^{t}(y)$ are finite for a.e. $(x,y) \in \domain$ and all $t \in I$. For those $(x, y) \in \domain$, if $t(x, u(x)) \leq t(y, u(y))$ then by property \ref{def:field:2} we have $w(x) = u^{t(x,w(x))}(x) \leq u^{t}(x)$ for $t \in [t(x, u(x)), t(y, u(y))]$, and the claim follows in this case.
	When $t(x, u(x)) \geq t(y, u(y))$ the argument is similar.
	
	The right-hand side of \eqref{eq:ex:8} can be integrated explicitly as
	\begin{equation}\label{eq:ex:9}
		\begin{split}
			-\frac{1}{2}\int^{t(y,w(y))}_{t(x, w(x))} (w(x)-u^{t}(y)) \partial_t{u}^t(y) \d t &= \frac{1}{4}\int^{t(y,w(y))}_{t(x, w(x))} \frac{\d}{\d t} |w(x) - u^{t}(y)|^{2} \d t\\
			&= \frac{1}{4}|w(x)-w(y)|^2 - \frac{1}{4}|w(x)-u^{t(x,w(x))}(y)|^2.
		\end{split}
	\end{equation}
	Now, using \eqref{eq:ex:8} and \eqref{eq:ex:9}  in \eqref{eq:ex:7}, it follows that
	\[
	\calibgageps(w) \leq \gagliardoeps(w).
	\]
	
	Finally, by property \ref{def:field:3}, $\fraclaplace_{\varepsilon} u^{t}$ converge to $\fraclaplace u^{t}$ in $L^1(\Omega)$, locally uniformly in $t$, as $\varepsilon \downarrow 0$.
	This is enough to pass to the limit in the inequality above and conclude the proof.
\end{proof}

We can finally give the proof of Theorem~\ref{thm:frac:calibration}.
We will show the identity
\begin{equation}
	\label{frac:calib:pot}
	\calibfrac(w) = \int_{\Omega} \int^{w(x)}_{u^{t_0}(x)} \big(\fraclaplace u^{t}(x) - F'(u^{t}(x))\big)\big|_{t = t(x, \lambda)} \d \lambda \d x + \energyfrac(u^{t_0})
\end{equation}
for $w \in \hbdy$, and that this functional is a calibration for $\energyfrac$ and each $u^{t}$ when the family $\{u^{t}\}_{t \in I}$ is a field of extremals, that is, when each $u^t$ solves the semilinear equation~\eqref{euler:lagrange:frac}.
In particular, each of the $u^t$ will be a minimizer.
More generally, we show that $u^{t_0}$ minimizes $\energyfrac$ if the $u^t$ above $u^{t_0}$ are supersolutions of~\eqref{euler:lagrange:frac} and the $u^t$ below are subsolutions.

\begin{proof}[Proof of Theorem~\ref{thm:frac:calibration}]
	First of all, note that for $x \in \Omega$ we have
	\[
	F(w(x)) - F(u^{t_0}(x)) =  \int_{u^{t_0}(x)}^{w(x)} F'(\lambda) \d \lambda = \int_{u^{t_0}(x)}^{w(x)} F'(u^{t(x,\lambda)}) \d \lambda 
	\]
	and, thanks to Remark~\ref{remark:pv}, the functional $\calibfrac$ given by \eqref{def:fraccalib}
	can be written simply as
	\[
	\calibfrac(w) = \calibgags(w) - \int_{\Omega} F(w(x)) \d x,
	\]
	where $\calibgags$ has been introduced in~\eqref{remark:pv}. This proves \eqref{frac:calib:pot}.
	
	{\rm(a)} 
	From~\eqref{frac:calib:pot} it is clear that $\calibfrac(u^{t_0}) = \energyfrac(u^{t_0})$, which is condition~\ref{def:calib:2}.
	To obtain~\ref{def:calib:3} we apply Lemma~\ref{lemma:inequality}, which gives
	\[
	\calibfrac(w) = \calibgags(w) - \int_{\Omega} F(w(x)) \d x \leq \gagliardo(u) - \int_{\Omega} F(w(x)) \d x = \energyfrac(w).
	\]
	
	{\rm(b)}
	To show~\ref{def:calib:1:prime}, that is, $\calibfrac(w) \geq \calibfrac(u^{t_0})=\energyfrac(u^{t_0})$, by~\eqref{frac:calib:pot} it suffices to show that
	\[
	\int^{w(x)}_{u^{t_0}(x)} \big(\fraclaplace u^{t}(x) - F'(u^{t}(x))\big)\big|_{t = t(x, \lambda)} \d \lambda \geq 0
	\]
	for $x \in \Omega$. But this is clear from the monotonicity of the field $u^{t}$ and the hypotheses in~\ref{frac:sup:sub}.
	
	On the other hand, we have already seen in subsection~\ref{subsection:calibration} how properties~\ref{def:calib:2},~\ref{def:calib:3}, and~\ref{def:calib:1:prime} yield the minimality of $u = u^{t_0}$.
	
	{\rm(c)}
	By~\eqref{frac:calib:pot}, using that each $u^{t}$ satisfies the Euler-Lagrange equation \eqref{euler:lagrange:frac}, we have that
	\[
	\calibfrac(w) - \energyfrac(u^{t_0}) = \int_{\Omega}\int_{u^{t_0}(x)}^{w(x)} \big(\fraclaplacian u^t(x) - F'(u^t(x))\big)\big|_{t = t(x, \lambda)} \d \lambda \d x = 0.
	\]
	Hence $\calibfrac(w) = \energyfrac(u^{t_0}) = \calibfrac(u^{t_0})$ for all $w \in \hbdy$.
	In particular, the functional $\calibfrac$ satisfies all three properties \ref{def:calib:2},\ref{def:calib:3}, and \ref{def:calib:1}, and thus it is a calibration.
	Choosing $t_0 = t$ for each $t \in I$, we deduce the minimality of $u^{t}$.
\end{proof}

\begin{proof}[Proof of Corollary~\ref{cor:applications}]
First of all, by a simple argument from~\cite{AlbertiAmbrosioCabre}, it suffices to prove the corollary for the class of functions $w \in C^0$ satisfying the strict inequality
\begin{equation}\label{limit:strict}
\lim_{\tau \to -\infty} u(x', \tau) < w(x', x_n) < \lim_{\tau \to +\infty} u(x', \tau) \quad \text{ for all } (x', x_n) \in \R^{n-1}\times \R.
\end{equation}
Indeed, if $w$ satisfies the non-strict inequality~\eqref{limit:condition}, then for $\theta \in (0,1)$ we consider $w_{\theta} := (1-\theta) u + \theta w$, which satisfies \eqref{limit:strict} by the strict monotonicity of $u$.
Hence, applying the corollary in the strict case, we have $\energyfrac(u) \leq \energyfrac(w_{\theta})$.
Letting $\theta \to 1^{-}$ yields the result for $w$.
Hence, we may assume \eqref{limit:strict}.

Since $u$ is bounded and $F \in C^{3}$, by regularity theory for the fractional Laplacian, we have that $u$ is at least $C^{2}$ and $\nabla u \in L^{\infty}(\R^n)$; see~\cite{RosOtonSerra}.
For each $t \in \R,$ we consider the family of translations $u^{t}(x) = u(x', x_n + t)$.
By the monotonicity and the regularity properties of $u$, the family $\{u^{t}\}_{t \in \R}$ is a field in the sense of Definition~\ref{def:field:frac}.
Moreover, by the translation invariance of the equation, it is a field of extremals.
Hence, we can apply Theorem~\ref{thm:frac:calibration} to conclude that $u$ is a minimizer in the set of admissible functions~$\admissiblefrac$ with $w = u$ in $\Omega^c$.
\end{proof}

\appendix

\section{A calibration for the extension problem of the fractional Laplacian}
\label{section:extension}

In this appendix we study minimizers of the energy functional $\mathcal{E}_{s, F}$ by using the extension technique for the fractional Laplacian.
The strategy is based on building a calibration for an auxiliary local energy $\widetilde{\mathcal{E}}_{s,F}$ in the extended space $\R_+^{n+1} = \R^n\times (0,\infty)$.
We point out that this construction did not give us, during the conception of our work, any a priori information about the form of a calibration written ``downstairs'' (i.e., in~$\R^n$) for the original energy functional $\mathcal{E}_{s, F}$. It was only after finding $\mathcal{C}_{s, F}$ by nonlocal arguments that we noticed how to deduce it, at least formally, from the extension problem. 

We denote by $(x,z)\in \R^n\times \R$ points in $\R^{n+1}_+$. Given a bounded domain $\Omega \subset \R^n$, we say that a bounded set $\widetilde{\Omega} \subset \R_+^{n+1}$ is an extension of $\Omega$ if $\partial_0\widetilde{\Omega} := \partial \widetilde{\Omega} \cap \{z=0\} \subset \Omega$. 
It is well known that there is a strong connection between the nonlocal energy functional~$\mathcal{E}_{s, F}$ and the local one
\begin{equation*}
	\widetilde{\mathcal{E}}_{s, F}(W;\widetilde{\Omega}) :=\frac{d_s}{2}\iint_{\widetilde{\Omega}} z^{1-2s} |\nabla W(x,z)|^2 \d x \d z - \int_{\partial_0\widetilde{\Omega}} F(W(x,0)) \d x,
\end{equation*}
where $d_s$ is a positive normalizing constant.
For this, given a function $u$ defined in $\R^n$ we consider $U$ the solution of
\begin{equation*}
	\beqc{\PDEsystem}
	\div(z^{1-2s} \nabla U) &=& 0 & \textrm{ in } \R^n\times (0,\infty)\,, \\
	U &=& u & \textrm{ on } \R^n\,.
	\eeqc
\end{equation*}
Here, $U$ is the so-called $s$-harmonic extension of $u$.
In~\cite[Lemma~7.2]{CaffarelliRoquejoffreSavin}, Caffarelli, Roquejoffre, and Savin showed that $u$ is a minimizer of $\mathcal{E}_{s, F}$ among functions with the same exterior data as $u$ in $\Omega^c$ if and only if, for every extension domain $\widetilde{\Omega}$, the $s$-harmonic extension $U$ of $u$ is a minimizer of $\widetilde{\mathcal{E}}_{s}(\cdot;\widetilde{\Omega})$ among functions with the same boundary condition as $U$ on $\partial_L \widetilde{\Omega} := \partial \widetilde{\Omega} \cap \{z>0\}$.

Taking into account this equivalence we can apply the classical theory of calibrations to the mixed Dirichlet-Neumann problem as explained in Remark~\ref{Remark:Dirichlet-Neumann}. 
To do this, given a field $\{u^{t}\}_{t \in I}$ in $\R^n$, for some interval $I \subset \R$, it is clear by the maximum principle that we can define a new field $\{U^{t}\}_{t \in I}$ in $\R^{n+1}_{+}$ where each leaf $U^t$ is the $s$-harmonic extension of $u^t$. 
Then, the functional
\begin{equation}\label{Ext_calib:loc}
	\begin{split}
		\widetilde{\mathcal{C}}_{s, F}(W;\widetilde{\Omega}) := &d_s\iint_{\widetilde{\Omega}} z^{1-2s} \Big\{\nabla U^t(x,z)\cdot \nabla W(x,z)-\frac{1}{2} |\nabla U^t(x,z)|^2   \Big\}_{t = t(x, z, W(x,z))} \d x \d z \\
		&\ \ \ - \int_{\partial_0\widetilde{\Omega}} F(W(x,0)) \d x
	\end{split}
\end{equation}
can be proved to be a calibration for $\widetilde{\mathcal{E}}_{s, F}$ and $U$. 
Therefore, $U$ is a minimizer of $\widetilde{\mathcal{E}}_{s, F}$, and by~\cite[Lemma~7.2]{CaffarelliRoquejoffreSavin} it follows that $u$ is a minimizer of $\mathcal{E}_{s, F}$.

We point out that although in this way we easily found a calibration for the local energy  $\widetilde{\mathcal{E}}_{s, F}(\cdot;\widetilde{\Omega})$, it was not clear at all how it translated into a calibration written ``downstairs'' for the original energy functional $\mathcal{E}_{s, F}$. 
It was only after building the calibration $\mathcal{C}_{s, F}$ by using purely nonlocal techniques that we discovered how to pass, at least formally, from $\widetilde{\mathcal{C}}_{s, F}(\cdot;\widetilde{\Omega})$ to $\mathcal{C}_{s, F}$. 
Let us explain this. 
First, as in Section~\ref{section:local}, for~$t_0 \in I$, we rewrite~\eqref{Ext_calib:loc} in the alternative form\footnote{Here, the first term is the one associated to the Euler-Lagrange operator of the local energy functional and vanishes by the definition of the field $U^t$. 
	On the other hand, the second and third terms are the ones involving the Neumann operator for the extended problem.}
\begin{equation*}
	\begin{split}
		\widetilde{\mathcal{C}}_{s, F}(W;\widetilde{\Omega}) &= - d_s \iint_{\widetilde{\Omega}} \int_{U^{t_0}(x,z)}^{W(x,z)} \div\left( z^{1-2s} \nabla U^t(x,z) \right)\Big|_{t = t(x, z, \lambda)}  \d \lambda \d z \d x \\
		&\ \ \ \ + \int_{\partial_0\widetilde{\Omega}} \int_{U^{t_0}(x,0)}^{W(x,0)} \Big\{ \fraclaplacian u^t(x)-F'(u^t(x))  \Big\}\Big|_{t = t(x, 0, \lambda)}  \d \lambda \d x \\
		&\ \ \ \ + d_s \int_{\partial_L\widetilde{\Omega}} \int_{U^{t_0}(x,z)}^{W(x,z)} z^{1-2s} \, \nu_{\partial_{L} \widetilde{\Omega}} \cdot \nabla U^t(x,z)\Big|_{t = t(x, z, \lambda)}  \d \lambda \d \mathcal{H}^{n}(x,z) \\
		&\ \ \ \ + \widetilde{\mathcal{E}}_{s, F}(U^{t_0};\widetilde{\Omega}),
	\end{split}
\end{equation*}
where $\nu_{\partial_{L} \widetilde{\Omega}}$ is the exterior normal vector to the lateral boundary $\partial_{L} \widetilde{\Omega}$. 
Finally, taking a sequence of extended domains $\widetilde{\Omega}_i$ converging to the half-space $\R^{n+1}_+$, we recover the functional $\mathcal{C}_{s, F}$ (up to an additive constant) as the formal limit of $\widetilde{\mathcal{C}}_{s, F}(\cdot;\widetilde{\Omega}_i)$.


\section{Other candidates for the fractional calibration}
\label{section:examples}
In this section we discuss three other natural candidates to be a calibration for the energy functional
$$
\energyfrac(w) = \dfrac{c_{n,s}}{4}\iint_{\domain} \dfrac{\abs{w(x)-w(y)}^2}{\abs{x-y}^{n+2s}} \d x \d y - \int_{\Omega} F(w(x)) \d x.
$$
We will be able to discard two of them since some of the calibration properties fail in these cases. 
Nevertheless, there is still one candidate for which we cannot determine whether it is a calibration or not.

Let us recall that the local counterpart of $\energyfrac$ is the functional
\[
\energyquad(w) = \dfrac{1}{2}\int_{\Omega} \abs{\nabla w(x)}^2 \d x - \int_{\Omega} F(w(x)) \d x,
\]
which admits the calibration
\[
\calibquad(w) = \int_{\Omega} \Big( \nabla u^{t}(x) \cdot (\nabla w(x)-\nabla u^{t}(x)) + \dfrac{1}{2} \abs{\nabla u^{t}(x)}^2 \Big)\Big|_{t = t(x, w(x))} \d x - \int_\Omega F(w(x)) \d x,
\]
a functional that can  also be written as
\[
\calibquad(w) = \int_{\Omega} \Big( \nabla u^{t}(x) \cdot \nabla w(x) - \dfrac{1}{2} \abs{\nabla u^{t}(x)}^2 \Big)\Big|_{t = t(x, w(x))} \d x - \int_\Omega F(w(x)) \d x.
\]

Inspired by the form of $\calibquad$, the first natural calibration candidate for $\energyfrac$ can be built replacing the gradient terms by differences and double integrals.
That is, we let
\[\begin{split}
	\mathcal{F}_{s, F}^{1}(w) &:= \frac{c_{n,s}}{2} \iint_{\domain}\dfrac{(u^{t}(x) - u^{t}(y))(w(x) - w(y))}{|x-y|^{n+2s}}\bigg|_{t = t(x, w(x))} \d x \d y \\
	& \quad \,\quad  - \dfrac{c_{n,s}}{4} \iint_{\domain} \dfrac{|u^{t}(x) - u^{t}(y)|^2}{|x-y|^{n+2s}}\bigg|_{t = t(x, w(x))} \d x \d y - \int_\Omega F(w(x)) \d x.
\end{split}\]
By using Young's inequality and the definition of the leaf-parameter function, one can directly conclude that $\mathcal{F}_{s, F}^{1}$ satisfies properties~\ref{def:calib:2} and~\ref{def:calib:3}.
It remains to check whether the null-Lagrangian property~\ref{def:calib:1} is satisfied, but we do not know how to answer this question.
For the affirmative answer, the idea would be to use the usual nonlocal integration by parts technique to obtain the Euler-Lagrange equation on the leaves.
However, since the leaf-parameter function $t$ depends on the variable $x$, we get remainder terms that we do not know how to treat.
It is then natural to look for a counterexample.
We looked at cases where an explicit field is available. 
For the trivial potential $F=0$, for which $u^t(x) = x+t$ are extremals (even if not bounded), property~\ref{def:calib:1} does not fail.
Hence, this case does not discard the candidate $\mathcal{F}_{s, F}^{1}$.
Another interesting example with explicit solutions is the Peierls-Nabarro model,
corresponding to the case $n = 1$, $s = 1/2$, and $F(u) = 1- \cos(u)$.
Here the equation $(-\Delta)^{1/2} u = \sin (u)$ in $\R$ admits the field of extremals $u^t(x) = 2 \arctan (x + t)$.
We do not know if the null-Lagrangian property holds for $\mathcal{F}_{s, F}^{1}$ in this concrete example.

It is also interesting to compare $\mathcal{F}_{s, F}^{1}$ with the calibration $\calibfrac$ constructed in Section~\ref{section:fractional:laplacian}.
There, by the alternative expression for $\calibfrac$ derived in Lemma~\ref{lemma:equality}, we see that~$\mathcal{F}_{s, F}^{1}(w)$  and $\calibfrac(w)$ would coincide if the following equality were true:
\[
\begin{split}
	&- \lim_{ \varepsilon \downarrow 0} \iint_{\domain \setminus \{|x-y| > \varepsilon\}} \dfrac{\int_{t(x,w(x))}^{t(y,w(y))} (u^{t}(x) - u^{t}(y))\partial_{t}u^{t}(y) \d t}{|x-y|^{n+2s}} \d x \d y \\
	&\hspace{3.5cm} = \iint_{\domain}\dfrac{(w(x) - u^{t(x,w(x))}(y))(u^{t(x,w(x))}(y) - w(y))}{|x-y|^{n+2s}} \d x \d y.
\end{split}
\]
However, we do not know how to prove or disprove this identity.

The functional $\mathcal{F}_{s, F}^{1}$ does not capture the symmetry in the variables $x$ and $y$ that has appeared in the two previous works on nonlocal calibrations~\cite{Cabre-Calibration, Pagliari}.
Hence, it is also natural to propose the following new candidate:
\[\begin{split}
	\mathcal{F}_{s, F}^{2}(w) &:= \frac{c_{n,s}}{2} \iint_{\domain}\dfrac{(u^{\tau}(x) - u^{t}(y))(w(x) - w(y))}{|x-y|^{n+2s}} \bigg|_{\substack{t = t(x, w(x))\\ \tau = t(y, w(y))}} \d x \d y \\
	& \quad \, \quad - \dfrac{c_{n,s}}{4} \iint_{\domain}\dfrac{|u^{r}(x) - u^{t}(y)|^2}{|x-y|^{n+2s}} \bigg|_{\substack{t = t(x, w(x))\\ \tau = t(y, w(y))}} \d x \d y - \int_\Omega F(w(x)) \d x.
\end{split}\]
As in the preceding case, we can apply Young's inequality and the definition of the leaf-parameter function to deduce that $\mathcal{F}_{s, F}^{2}$ satisfies properties~\ref{def:calib:2} and~\ref{def:calib:3}.
Nevertheless, in this case we can discard it as a calibration since the null-Lagrangian property fails even when $F=0$ and $u^t(x) = x+t$.

One could also think of a calibration candidate constructed by replacing the gradient terms in the local theory by fractional ones.
That is,
$$ \mathcal{F}_{s, F}^3(w) := \int_{\Omega} \Big\{\nabla^s u^{t}(x) \cdot \nabla^s w(x) \d x - \dfrac{1}{2} \abs{\nabla^s u^{t}(x)}^2 \Big\}\Big|_{t = t(x, w(x))} \d x - \int_\Omega F(w(x)) \d x. $$
Here, the fractional gradient is defined as
\[
\nabla^s w (x) = \widetilde{c}_{n,s}\int_{\R^n} \frac{w(x)-w(y)}{|x-y|^{n+s}}\dfrac{x-y}{\abs{x-y}}\,dy.
\]
This last candidate would be motivated by the identity 
\[
\int_{\R^n} \nabla^s v(x) \cdot \nabla^s w(x) \d x = \dfrac{c_{n,s}}{2}\iint_{\R^n\times \R^n} \dfrac{(v(x)-v(y))(w(x)-w(y))}{\abs{x-y}^{n+2s}} \d x \d y.
\]
Nevertheless, a similar equality does not hold when restricting to a domain $\Omega$, i.e.,
\[
\int_{\Omega} \nabla^s v(x) \cdot \nabla^s w(x) \d x \neq \dfrac{c_{n,s}}{2}\iint_{\domain} \dfrac{(v(x)-v(y))(w(x)-w(y))}{\abs{x-y}^{n+2s}} \d x \d y.
\]
Hence, $\mathcal{F}_{s, F}^3$ does not satisfy property~\ref{def:calib:2} and thus it is not a calibration for $\energyfrac$.


\end{document}